\documentclass[leqno,final]{article}

\usepackage{graphicx}
\usepackage{amsmath,amstext,amssymb,amsthm}
\newtheorem{theorem}{Theorem}[section]
\newtheorem{lemma}[theorem]{Lemma}
\newtheorem{corollary}[theorem]{Corollary}
\newtheorem{definition}{Definition}[section]

\newtheorem{remark}{Remark}[section]
\newcommand{\norm}[1]{\left\Vert#1\right\Vert}
\newcommand{\norml}[2]{\left\Vert#1\right\Vert_{L^2(#2)}}
\newcommand{\abs}[1]{\bigl\vert#1\bigr\vert}
\newcommand{\pd}[1]{\left\langle #1\right\rangle}

\newcommand{\pdjj}[2]{\left\langle \left[#1\right], \left[#2\right]\right\rangle_e}
\newcommand{\set}[1]{\left\{#1\right\}}

\newcommand{\jm}[1]{\left[#1\right]}

\newcommand{\tuh}{\tilde{u}_h}
\newcommand{\csta}{C_{\rm sta}}

\newcommand{\db}{\displaybreak[0]}
\newcommand{\nn}{\nonumber}

\newcommand{\De}{\Delta}
\newcommand{\ep}{\varepsilon}
\newcommand{\ga}{\gamma}
\newcommand{\Ga}{\Gamma}

\newcommand{\La}{\Lambda}
\newcommand{\na}{\nabla}
\newcommand{\om}{\omega}
\newcommand{\Om}{\Omega}
\newcommand{\pa}{\partial}
\newcommand{\pr}{\prime}

\newcommand{\ta}{\theta}
\newcommand{\vp}{\varphi}

\renewcommand{\i}{{\rm\mathbf i}}

\DeclareMathOperator{\re}{{Re}}
\DeclareMathOperator{\im}{{Im}}

\newcommand{\bR}{\mathbf{R}}
\newcommand{\p}{\partial}
\newcommand{\cT}{\mathcal{T}}
\newcommand{\cE}{\mathcal{E}}

\newcommand{\hl}{\underline{h}}
\newcommand{\gal}{\underline{\gamma}}
\newcommand{\uhfem}{u_h^{\rm FEM}}

\title{Continuous Interior Penalty Finite Element Methods for the Helmholtz Equation with Large
Wave Number}
\author{
Haijun Wu
\thanks{Department of Mathematics, Nanjing University, Jiangsu,
210093, P.R. China. ({\tt hjw@nju.edu.cn}).}}

\begin{document}
\date{}
\maketitle


\setcounter{page}{1}

\begin{abstract}
This paper develops and analyzes some continuous interior penalty finite element methods (CIP-FEMs) using piecewise linear polynomials for the Helmholtz equation with
the first order absorbing boundary condition in two and three dimensions. The novelty of the proposed methods is to use complex penalty parameters with positive imaginary parts.
It is proved that, if the penalty parameter is a pure imaginary number $\i\ga$ with $0<\ga\le C$, then the proposed CIP-FEM is stable (hence well-posed) without any mesh constraint. Moreover the method  satisfies the error estimates $C_1kh+C_2k^3h^2$ in the $H^1$-norm when $k^3h^2\le C_0$ and $C_1kh+\frac{C_2}{\ga}$ when $k^3h^2> C_0$ and $kh$ is bounded, where $k$ is the wave number, $h$ is the mesh size, and the $C$'s are positive constants independent of $k$, $h$, and $\ga$. Optimal order $L^2$ error estimates are also derived. The analysis is also applied if the penalty parameter is a complex number with positive imaginary part. By taking $\ga\to 0+$, the above estimates are extended to the linear finite element method under the condition $k^3h^2\le C_0$. Numerical results are provided to verify the theoretical findings. It is shown that the penalty parameters may be tuned to greatly reduce the pollution errors.  
\end{abstract}




\section{Introduction}\label{sec-1} The  problem of short waves (or waves  with high wave numbers) in acoustics, electromagnetics or surface water wave applications was listed as an unsolved problem in finite element methods (FEMs) in 2000 by Zienkiewicz in his review paper \cite{zienkiewicz00}. It still remains open although some big progresses have been made since then. In this paper,
we consider the following Helmholtz problem:
\begin{alignat}{2}
-\De u - k^2 u &=f  &&\qquad\mbox{in  } \Om,\label{e1.1}\\
\frac{\pa u}{\pa n} +\i k u &=g &&\qquad\mbox{on } \Ga,\label{e1.2}
\end{alignat}
where $\Om\subset \bR^d,\, d=2,3$ is a polygonal/polyhedral
domain,
$\Ga:=\pa\Om$, $\i=\sqrt{-1}$ denotes the imaginary unit, and $n$
denotes the unit outward normal
to $\pa\Om$. The above Helmholtz problem is an approximation of the  following
acoustic scattering problem (with time dependence $e^{\i\om t}$):
\begin{alignat}{2}\label{e1.1a}
-\De u- k^2 u &= f \qquad &&\mbox{in } \bR^d,\\
\sqrt{r}\Bigl( \frac{\p (u-u^{\rm inc})}{\p r} + \i k (u-u^{\rm inc})\Bigr) &\rightarrow 0 &&\mbox{as }
r=|x|\rightarrow \infty, \label{e1.2a}
\end{alignat}
where  $u^{\rm inc}$ is the incident wave and $k$ is known as the wave number. The Robin boundary condition \eqref{e1.2} is known as the
first order approximation of the radiation condition \eqref{e1.2a} (cf. \cite{em79}).
We remark that the Helmholtz problem \eqref{e1.1}--\eqref{e1.2} also arises in applications
 as a consequence of frequency domain treatment of attenuated scalar waves 
(cf. \cite{dss94}).  

The difficulties of FEMs applied to the Helmholtz problem \eqref{e1.1}--\eqref{e1.2} with large wave number lie in both their theoretical analysis and numerical efficiency mainly due to the high indefiniteness of the problem. While for the one dimensional ($1$-D) case the FEMs have been well understood.  Ihlenburg and Babu{\v{s}}ka \cite{ib95a} proved that the linear FEM for a $1$-D Helmholtz problem satisfies the following error estimate under the mesh constraint $kh\le 1$.
\begin{equation}\label{e1.3}
\norml{\na(u-\uhfem)}{\Om}\le C_1kh+C_2k^3h^2.
\end{equation}
Here $h$ is the mesh size and $C_i, i=1,2$ are positive constants independent of $k$ and $h$.
Note that the first term on the right hand side of \eqref{e1.3} is of the same order as the interpolation error in $H^1$-seminorm. It dominates the error bound only if $k^2h$ is small enough. The second term $C_2k(kh)^2$ dominates if $kh$ is fixed and $k$ is large enough. We remark that the condition of fixed $kh$, i.e., several points per wavelength, is  sometimes used as the ``rule of thumb" in the context of the numerical treatment of the Helmholtz equation. The estimate \eqref{e1.3} says that this rule of thumb may give wrong results for large wave number $k$. The second term is called the pollution error of the finite element solution.  In one dimension, the pollution effect can be eliminated completely by a suitable modification of the discrete bilinear form (cf. \cite{bips95,bs00}). However, the story for two and three dimensional Helmholtz problems is much different. It is shown that, in two dimensions, the pollution effect can be reduced substantially but cannot be avoided in principle (cf. \cite{bips95,bs00}). As for the error estimates, to the best of the author's knowledge, no analysis for the linear FEM in two or three dimensions has been done when $k^2h$ is large.  Note that the $H^1$- and $L^2$- error estimates can be derived by the so-called Schatz argument if $k^2h$ is small enough  (cf. \cite{bgt85, sch74}) but this condition is too strict for large $k$. We refer to \cite{thompson06} for a nice review on various FEMs for time-harmonic acoustics governed by the Helmholtz equation. For results on $hp$-FEMs, we refer to \cite{ms10} but will not discuss here since we concern only methods using linear elements in this paper. The author would like to mention that, Engquist and Ying \cite{ey11,ey10_arxiv} proposed recently some sweeping preconditioners for central difference schemes for the Helmholtz equation which have linear application cost and the preconditioned iterative solver (GMRES) converges in a number of iterations that is essentially independent of the number of unknowns or the frequency. Although the sweeping techniques are well possible (or have already been) applied to the linear FEM to provide efficient fast solvers, this combination is a not good candidate for efficient algorithm for the Helmholtz problem with large wave number, since the linear FEM itself is inefficient due to its (big) pollution effect. Next, we will not go any further on the issue of fast solvers and focus on the stability and error analyses of the schemes based on linear elements.    

In \cite{fw09, fw11}, Feng and the author proposed and analyzed some interior penalty discontinuous Galerkin (IPDG) methods using piecewise linear polynomials for the problem \eqref{e1.1}--\eqref{e1.2} in two and three dimensions.  It was proved that the proposed methods are unconditionally (with respect to mesh size $h$) stable and well-posed for all wave numbers $k > 0$. Moreover, under suitable assumptions on the penalty parameters, the following error estimates were proved.
  \begin{align*}
   \norm{u-u_h}_{1,h}&\le C_1kh+C_2k^{8/3}h^{4/3},&\text{ if }  kh\lesssim 1,\\
   \norm{u-u_h}_{1,h}&\le C_1kh+C_2k^{3}h^{2},&\text{ if }  k^3h^2\le C_0.
  \end{align*}
where $\norm{\cdot}_{1,h}$ is some broken $H^1$-norm, $A\lesssim B$ means $A\le C\,B$, and the constants $C$'s are positive and independent of $k$, $h$, and the penalty parameters. Numerical tests show that it is possible to greatly reduce the pollution error and achieve better numerical results than FEMs by tuning the penalty parameters (see \cite{fw09}). The discontinuous Galerkin (DG) methods (sometimes called discontinuous finite element methods) which are initiated in seventies of the last century (cf. \cite{b73,bz73,baker77,dd76, w78, arnold82}), use piecewise polynomials (or problem dependent functions) as trial and test functions. The continuity of the discrete solution across the interior edges/faces of elements is enforce weakly by introducing penalty terms or numerical fluxes.
As is well known now, DG methods have several advantages over the (continuous) FEMs, such as, local mass conservation, flexibilities in constructing trial and test spaces and meshes, additional parameters that may be tuned for some particular purposes. While one disadvantage is that a DG method usually has larger number of total degrees of freedom (DOFs) than the FEM. For example, on a given triangulation of $\Om$, the number of total DOFs of the linear IPDG method is about six times of that of the linear FEM in two dimensions and about $20$ times in three dimensions. We refer the reader to \cite{aldr06, acdg10, fx, gm, perugia07, dpg11} and the references therein for other works on DG methods for Helmholtz problems.

The purpose of this paper is to propose and analyze a linear continuous interior penalty finite element method (CIP-FEM) for the Helmholtz problem \eqref{e1.1}--\eqref{e1.2}. The CIP-FEM uses the same continuous piecewise linear finite element space as the linear FEM but modifies the sesquilinear of the FEM by adding a penalty term on the jumps of the flux across the interior edges/faces between elements, i.e., 
\begin{align*}
J(u,v):=\i\ga\sum_{e\in\cE_h^{I}} h_e \int_e\jm{\frac{\pa u}{\pa n_e}}\jm{\frac{\pa v}{\pa n_e}},
\label{eJ}\db
\end{align*}
where $\ga>0$ and $\cE_h^I$ is the set of interior edges/faces. Note that the CIP-FEM in this paper uses a pure-imaginary penalty parameter $\i\ga$ instead of a real one as the usual CIP-FEM does. This is helpful for theoretical analysis and numerical stability. It is should be remark that if the penalty parameter $\i\ga$
 is replaced by a complex number with positive imaginary part, the ideas of the paper still apply.
Here we set its
real part to be zero in the theoretical analysis for the ease of presentation. Let $u_h$ be the CIP finite element solution and let $\uhfem$ be the finite element solution. The following results are obtained.
\begin{itemize}
\item[(i)] The CIP-FEM attains a unique solution for any $k>0$, $h>0$ and $\ga>0$.
\item[(ii)]  There exists
a constant $C_0>0$ independent of $k$, $h$, and $\ga$, 
such that if $k\gtrsim 1$ and
$0<\ga\lesssim 1$, then the following stability and error
estimates hold:
\begin{align*}
\norm{u_h}_{1,h}&\lesssim  \left\{\begin{array}{ll}
				\norml{f}{\Om} + \norml{g}{\Ga}, &\text{ if }k^3 h^2\le  C_0,\\
				\dfrac{1}{\ga}(\norml{f}{\Om} + \norml{g}{\Ga}), &\text{ if }k^3 h^2>  C_0,
				\end{array}\right.\\
\norm{u-u_h}_{1,h} & \le\left\{\begin{array}{ll}
				C_1 kh+C_2 k^3h^2, &\text{ if }k^3 h^2\le  C_0,\\
				C_1 kh+\dfrac{C_2}{\ga}, &\text{ if }k^3 h^2>  C_0 \text{ and } k h\lesssim 1,
				\end{array}\right.
\end{align*}
where $\norm{v}_{1,h}:=\big( \norml{\na v}{\Om}^2+\abs{J(v,v)}\big)^{1/2}$.
\item[(iii)] Suppose $k^3 h^2\le  C_0$ and $k\gtrsim 1$. Then the following estimates hold for the finite element solution $\uhfem$.
\begin{align*}
\norml{\na\uhfem}{\Om}&\lesssim  (\norml{f}{\Om} + \norml{g}{\Ga}),\\
\norml{\na(u-\uhfem)}{\Om} &\le C_1 kh+C_2 k^3h^2,
\end{align*}
\item[(iv)] Estimates in the $L^2$-norm are also obtained.
\item[(v)]  Numerical tests show that the penalty parameters may be tuned to greatly reduce the pollution errors. 
\end{itemize}
The CIP-FEMs were originally proposed by Douglas and Dupont \cite{dd76} for second order elliptic and parabolic problems and have been shown to have advantages for advection dominated problems \cite{burman05, be05,be07, bfh06, bh04}. Similar interior penalty procedures for FEMs utilizing continuous functions have also been introduced for biharmonic equations \cite[etc]{bz73,bs05}. 

This paper is organized as follows. The CIP-FEM is introduced in Section~\ref{sec-2}. Some stability estimates are derived in Section~\ref{sec-sta} for any $k>0$, $h>0$, and $\ga>0$. In Section~\ref{sec-err}, pre-asymptotic error estimates in $H^1$- and $L^2$-norms are proved for $k>0$, $h>0$, and $\ga>0$ by utilizing the error analysis for an elliptic projection, the stability results for the CIP-FEM, and the triangle inequality. In Section~\ref{sec-SEII}, the stability estimates in Section~\ref{sec-sta} and the error estimates in Section~\ref{sec-err} are improved to be of optimal order under the condition that $k^3h^2$ is small enough by using the technique of so-called  
``stability-error iterative improvement" developed in \cite{fw11}. In Section~\ref{sec-fem}, 
the well-posedness, stability and error estimates for the linear FEM are established  under the condition that $k^3h^2$  is small enough by taking the limits of the estimates for the CIP-FEM as the parameter $\ga\to 0+$.  

Throughout the paper, $C$ is used to denote a generic positive constant
which is independent of $h$, $k$, and the penalty parameters. We also use the shorthand
notation $A\lesssim B$ and $B\gtrsim A$ for the
inequality $A\leq C B$ and $B\geq CA$. $A\simeq B$ is a shorthand
notation for the statement $A\lesssim B$ and $B\lesssim A$. We assume that $k\gtrsim 1$ since we are considering high-frequency problems. For the ease of presentation, we assume that $k$ is constant on the domain $\Om$.

\section{Formulation of continuous interior penalty finite element methods}\label{sec-2}
To formulate our CIP-FEMs, we first introduce some notation. The standard space, norm and inner product notation
are adopted. Their definitions can be found in \cite{bs08,ciarlet78}.
In particular, $(\cdot,\cdot)_Q$ and $\pd{ \cdot,\cdot}_\Sigma$
for $\Sigma\subset \pa Q$ denote the $L^2$-inner product
on complex-valued $L^2(Q)$ and $L^2(\Sigma)$
spaces, respectively. Denote by $(\cdot,\cdot):=(\cdot,\cdot)_\Om$
and $\pd{ \cdot,\cdot}:=\pd{ \cdot,\cdot}_{\p\Om}$. 

Let $\cT_h$ be a family of triangulations of the domain $\Om$
parameterized by $h>0$. For any triangle/tetrahedron $K\in \cT_h$, we define $h_K:=\mbox{diam}(K)$. 
Similarly, for each edge/face $e$ of $K\in \cT_h$, define $h_e:=\mbox{diam}(e)$. Let $h=\max_{K\in\cT_h}h_K$.
We assume that the elements of $\cT_h$ are shape regular. We define
\begin{eqnarray*}
\cE_h^I&:=& \mbox{ set of all interior edges/faces of $\cT_h$},\\
\cE_h^B&:=& \mbox{ set of all boundary edges/faces of $\cT_h$ on $\Ga$}.
\end{eqnarray*}
We also define the jump $\jm{v}$ of $v$ on an interior edge/face
$e=\p K\cap \p K^\pr$ as
\[
\jm{v}|_{e}:=\left\{\begin{array}{ll}
       v|_{K}-v|_{K^\pr}, &\quad\mbox{if the global label of $K$ is bigger},\\
       v|_{K^\pr}-v|_{K}, &\quad\mbox{if the global label of $K^\pr$ is bigger}.
\end{array} \right.
\]
 For every $e=\p K\cap \p K^\pr\in\cE_h^I$, let $n_e$ be the unit outward normal
to edge/face $e$ of the element $K$ if the global label of $K$ is bigger
and of the element $K^\pr$ if the other way around. For every $e\in\cE_h^B$, let $n_e=n$ the unit outward normal to $\pa\Om$.

Now we define the ``energy" space $E$ and the sesquilinear
form $a_h(\cdot,\cdot)$ on $E\times E$ as follows:
\begin{align}
E&:=H^1(\Om)\cap\prod_{K\in\cT_h} H^2(K), \nonumber \\
\label{eah}
a_h(u,v)&:=(\na u,\na v)+  J(u,v)\qquad\forall\, u, v\in E,
\end{align}
where
\begin{align}
J(u,v):=&\sum_{e\in\cE_h^{I}}\i\ga_e h_e \pdjj{\frac{\pa u}{\pa n_e}}{\frac{\pa v}{\pa n_e}},
\label{eJ}\db
\end{align}
and  $\ga_e, e\in\cE_h^I$ are 
nonnegative numbers to be specified later. 

\begin{remark}
(a) The terms in $J(u,v)$ are so-called penalty terms.
The penalty parameter in $J(u,v)$ is
$\i\ga_e$.  So it is a
pure imaginary number with positive imaginary part. It turns out that if it is replaced by a complex number with positive
imaginary part, the ideas of the paper still apply.  Here we set their
real parts to be zero partly because the terms from real parts do not help much
(and do not cause any problem either) in our theoretical analysis and partly for
the ease of presentation. On the other hand, our numerical experiments
in Section \ref{sec-num} indicate that using penalty parameters
with nonzero real parts helps to reduce the pollution effect in the error.

(b) Penalizing the jumps of normal derivatives 
 was used early by Douglas and Dupont \cite{dd76} for second order PDEs and by Babu{\v{s}}ka and Zl\'amal \cite{bz73} for fourth order PDEs in
the context of $C^0$ finite element methods, by Baker \cite{baker77} for fourth order PDEs
and by Arnold \cite{arnold82} for second order
parabolic PDEs in the context of IPDG methods. 

(c) In this paper we consider the scattering problem with time dependence $e^{\i\om t}$, that is, the signs before $\i$'s in the Sommerfeld radiation condition \eqref{e1.2a} and its first order approximation \eqref{e1.2} are positive. If we consider the scattering problem with time dependence $e^{-\i\om t}$, that is, the signs before $\i$'s in  \eqref{e1.2a} and  \eqref{e1.2} are negative, then the penalty parameters should be complex numbers with  negative imaginary parts.
\end{remark}

It is clear that $J(u,v)=0$ if $u\in H^2(\Om)$ and $v\in E$. Therefore, if $u\in H^2(\Om)$ is the solution of \eqref{e1.1}--\eqref{e1.2}, then 
\begin{equation}
a_h(u,v) - k^2(u,v) +\i k \pd{ u,v}
=(f,v)+\pd{g, v},\qquad\forall v\in E.
\label{2.6}
\end{equation}

Let $V_h$ be the linear finite element space, that is,
\[
V_h:=\set{v_h\in H^1(\Om) :\; v_h|_K\in P_1(K),\,\forall K\in \cT_h}.
\]
where $P_1(K)$ denote the set of all linear
polynomials on $K$. Then our CIP-FEMs are defined as follows
: Find $u_h\in V_h$ such that
\begin{equation}\label{eicg}
a_h(u_h,v_h) - k^2(u_h,v_h) +\i k \pd{ u_h,v_h}
=(f, v_h)+\pd{g, v_h},  \qquad\forall v_h\in V_h.
\end{equation}

The following semi-norm on the space $E$ is useful for the subsequent analysis:
\begin{align}
\norm{v}_{1,h}:=&\bigg( \norml{\na v}{\Om}^2+\sum_{e\in\cE_h^{I}}\ga_e h_e \norml{\jm{\frac{\pa v}{\pa n_e}}}{e}^2\bigg)^{1/2}.
 \label{e2.5}
\end{align}

In the next three sections, we shall consider the stability and
error analysis for the above CIP-FEMs. Especially, we are interested
in knowing how the stability constants and error constants depend on the wave
number $k$ (and mesh size $h$, of course) and what are the ``optimal" relationship
between mesh size $h$ and the wave number $k$. For the ease of presentation, we assume that $\ga_e\simeq \ga$ for some positive constant $\ga$ and that $h_K\simeq h$. 

\section{Stability estimates}\label{sec-sta} We first recall the stability estimates for the original Helmholtz problem \eqref{e1.1}--\eqref{e1.2} (cf. \cite{cf06,hetmaniuk07}).
\begin{theorem}\label{tstau}
Suppose $\Om\subset \bR^d$ is a strictly star-shaped domain. Then the solution
$u$ to the problem \eqref{e1.1}--\eqref{e1.2} satisfies
\begin{eqnarray}\label{e2.1}
\|u\|_{H^j(\Om)} \lesssim k^{j-1} 
\bigl( \norml{f}{\Om} + \norml{g}{\Ga} \bigr)
\end{eqnarray}
for $j=0, 1$ if $u\in H^{3/2+\ep}(\Om)$ for some $\ep>0$.  \eqref{e2.1} also holds for $j=2$ if $u\in H^2(\Om).$
\end{theorem}
The key idea in their analysis is to test \eqref{e1.1} by $v=u$ and $v=(x-x_\Om)\cdot\na u$, respectively, and use the Rellich identity (for the Laplacian), where $x_\Om$ is a point such that the domain $\Om$ is strictly star-shaped with respect to it. The idea has been successfully applied to the discontinuous Galerkin methods (cf. \cite{fw09,fw11, fx}) and to the spectral-Galerkin methods (cf. \cite{sw07}). As for our CIP-FEMs \eqref{eicg}, although the test function $v_h=u_h$ can still be used, the test function $v_h=(x-x_\Om)\cdot\na u_h$ does not apply since it is discontinuous and hence not in the test space $V_h$. For stability results for other types of boundary conditions we refer to \cite{cm08,ms10}.

Next, we derive stability estimates for the CIP-FEMs \eqref{eicg}. Note that $u_h$ is piecewise linear on $\cT_h$ and hence $\De u_h=0$ on each element $K
\in \cT_h$. We will show, by using integration by parts elementwisely, that $\norml{\na u_h}{\Om}$ may be bounded by the jumps of $\frac{\pa u_h}{\pa n_e}$ across each interior edge/face $e\in\cE_h^I$ and the $L^2(\Om)$-norm and the $L^2(\Ga)$-norm of $u_h$. Moreover the coefficient before $\norml{u_h}{\Om}$ can be controlled. On the other hand, by taking the test function $v_h=u_h$ in \eqref{eicg}, we may derive some reverse inequalities, that is, bound the jumps of $\frac{\pa u_h}{\pa n_e}$ across $e\in\cE_h^I$ and the $L^2$ norms of $u_h$ by $\norml{\na u_h}{\Om}$ and the given data. Then the desire stability estimates follow by combining them.

 We first bound $\norml{\na u_h}{\Om}$ by using integration by parts on each element. 
\begin{lemma}\label{lem3.2}  For any $0<\ep<1$, there exists a constant $c_\ep$ such that
\begin{equation*}
\norml{\na u_h}{\Om}^2\le \ep k^2\norml{u_h}{\Om}^2+ \frac{c_\ep}{k h}\,k\norml{u_h}{\Ga}^2+\frac{c_\ep}{k^2 h^2\ga}\sum_{e\in\cE_h^I}\ga_e h_e\norml{\jm{\frac{\pa u_h}{\pa n_e}}}{e}^2 .
\end{equation*}
\end{lemma}
\begin{proof}
Noting that $u_h$ is piecewise linear, we have
\begin{align*}
\norml{\na u_h}{\Om}^2=&\sum_{K\in\cT_h}\int_K\abs{\na u_h}^2=\sum_{K\in\cT_h}\int_{\pa K}\frac{\pa u_h}{\pa n} u_h\\
=&\sum_{e\in\cE_h^B}\int_e\frac{\pa u_h}{\pa n_e} u_h+\sum_{e\in\cE_h^I}\int_e\jm{\frac{\pa u_h}{\pa n_e}} u_h\\
\le&\sum_{e\in\cE_h^B}\norml{\na u_h}{e}\norml{u_h}{e}+\sum_{e\in\cE_h^I}\norml{\jm{\frac{\pa u_h}{\pa n_e}}}{e}\norml{u_h}{e}.
\end{align*}
For any edge/face $e\in\cE_h$, let $K_e\in\cT_h$ be an element containing $e$. From the trace inequality and the inverse inequality,
\begin{align*}
&\norml{\na u_h}{\Om}^2
\le C\sum_{e\in\cE_h^B}h_e^{-1/2}\norml{\na u_h}{K_e}\norml{u_h}{e}+C\sum_{e\in\cE_h^I}h_e^{-1/2}\norml{\jm{\frac{\pa u_h}{\pa n_e}}}{e}\norml{u_h}{K_e}\\
&\quad\le C h^{-1/2}\norml{\na u_h}{\Om}\norml{u_h}{\Ga}+C\ga^{-1/2} h^{-1}\bigg(\sum_{e\in\cE_h^I}\ga_e h_e\norml{\jm{\frac{\pa u_h}{\pa n_e}}}{e}^2\bigg)^{1/2}\norml{u_h}{\Om}\\&\quad\le \ep\norml{\na u_h}{\Om}^2+ \frac{C}{\ep k h}\,k\norml{u_h}{\Ga}^2\\
&\qquad+\ep(1-\ep)k^2\norml{u_h}{\Om}^2+\frac{C}{\ep(1-\ep) k^2 h^2\ga}\sum_{e\in\cE_h^I}\ga_e h_e\norml{\jm{\frac{\pa u_h}{\pa n_e}}}{e}^2
\end{align*}
which implies that Lemma~\ref{lem3.2} holds.
\end{proof}

Then we derive some reverse inequalities by taking $v_h=u_h$ in \eqref{eicg}.
\begin{lemma}\label{lem3.1}
Let $u_h\in V_h$ solve \eqref{eicg}. Then,
\begin{align}\label{e3.1}
&k^2 \norml{u_h}{\Om}^2\le 2\norml{\na u_h}{\Om}^2 +\frac{C}{k^2}\norml{f}{\Om}^2+\frac{C}{k}\norml{g}{\Ga}^2,\\
&\sum_{e\in\cE_h^I}\ga_e h_e\norml{\jm{\frac{\pa u_h}{\pa n_e}}}{e}^2+k\norml{u_h}{\Ga}^2\label{e3.2} \\
&\qquad\qquad\le \frac{C}{k}\norml{f}{\Om}\norml{\na u_h}{\Om} +\frac{C}{k^2}\norml{f}{\Om}^2+\frac{C}{k}\norml{g}{\Ga}^2. \nn
\end{align}
\end{lemma}
\begin{proof}
Taking $v_h=u_h$ in \eqref{eicg} yields
\begin{equation}\label{euh}
a_h(u_h,u_h) -k^2 \norml{u_h}{\Om}^2 + \i k \norml{u_h}{\Ga}^2
=(f, u_h)+\pd{g, u_h}.
\end{equation}
Therefore, by taking real part and imaginary part of the above equation
 we get 
\begin{align}\label{elem3.1a}
&k^2 \norml{u_h}{\Om}^2-\norml{\na u_h}{\Om}^2 \leq \abs{(f,u_h)+\pd{g, u_h} },\\
&\sum_{e\in\cE_h^I}\ga_e h_e\norml{\jm{\frac{\pa u_h}{\pa n_e}}}{e}^2+k\norml{u_h}{\Ga}^2\leq \bigl|(f,u_h)+\pd{g, u_h} \bigr|. \label{elem3.1b} 
\end{align}
From \eqref{elem3.1b},
\begin{align*}
\sum_{e\in\cE_h^I}\ga_e h_e\norml{\jm{\frac{\pa u_h}{\pa n_e}}}{e}^2+k\norml{u_h}{\Ga}^2\leq \norml{f}{\Om}\norml{u_h}{\Om}+\frac{1}{2k}\norml{g}{\Ga}^2+\frac{k}{2}\norml{u_h}{\Ga}^2
\end{align*}
which implies
\begin{equation}\label{elem3.1c}
\sum_{e\in\cE_h^I}\ga_e h_e\norml{\jm{\frac{\pa u_h}{\pa n_e}}}{e}^2+\frac{k}{2}\norml{u_h}{\Ga}^2\leq \norml{f}{\Om}\norml{u_h}{\Om}+\frac{1}{2k}\norml{g}{\Ga}^2.
\end{equation}
On the other hand, from \eqref{elem3.1a},
\begin{align*}
k^2 \norml{u_h}{\Om}^2\le \norml{\na u_h}{\Om}^2 +  \norml{f}{\Om}\norml{u_h}{\Om}+\frac{1}{2k}\norml{g}{\Ga}^2+\frac{k}{2}\norml{u_h}{\Ga}^2.
\end{align*}
By combining the above two estimates, we conclude that
\begin{align*}
k^2 \norml{u_h}{\Om}^2\le &\norml{\na u_h}{\Om}^2 +  2\norml{f}{\Om}\norml{u_h}{\Om}+\frac{1}{k}\norml{g}{\Ga}^2\\
\le& \norml{\na u_h}{\Om}^2 +  \frac{2}{k^2}\norml{f}{\Om}^2+\frac{k^2}{2}\norml{u_h}{\Om}+\frac{1}{k}\norml{g}{\Ga}^2
\end{align*}
which implies that \eqref{e3.1} holds. 

Plugging  \eqref{e3.1} into the right hand side of \eqref{elem3.1c} yields
\begin{align*}
&\sum_{e\in\cE_h^I}\ga_e h_e\norml{\jm{\frac{\pa u_h}{\pa n_e}}}{e}^2+\frac{k}{2}\norml{u_h}{\Ga}^2\\
&\leq \frac{C}{k}\norml{f}{\Om}\Big(\norml{\na u_h}{\Om}+\frac{1}{k}\norml{f}{\Om}+\frac{1}{k^{1/2}}\norml{g}{\Ga}\Big)+\frac{1}{2k}\norml{g}{\Ga}^2\\
&\leq \frac{C}{k}\norml{f}{\Om}\norml{\na u_h}{\Om}+\frac{C}{k^2}\norml{f}{\Om}^2+\frac{C}{k}\norml{g}{\Ga}^2.
\end{align*}
That is, \eqref{e3.2} holds. This completes the proof of the lemma.
\end{proof}

By combining Lemma~\ref{lem3.2} and Lemma~\ref{lem3.1} we may derive the following stability estimates for the CIP-FEMs.
\begin{theorem}\label{tsta}
The  solution $u_h\in V_h$  to the scheme \eqref{eicg} satisfies the following stability estimates.
\begin{align}\label{e3.3}
&\norml{\na u_h}{\Om}^2+k^2 \norml{u_h}{\Om}^2\lesssim\csta^2\norml{f}{\Om}^2+\csta\norml{g}{\Ga}^2,\\
&\sum_{e\in\cE_h^I}\ga_e h_e\norml{\jm{\frac{\pa u_h}{\pa n_e}}}{e}^2+k\norml{u_h}{\Ga}^2\label{e3.4}\lesssim \frac{\csta}{k}\norml{f}{\Om}^2+\frac{1}{k}\norml{g}{\Ga}^2.
\end{align}
Here
\[\csta:=\frac1k+\frac{1}{k^2 h}+\frac{1}{ k^3 h^2\ga}.\]
\end{theorem}
\begin{proof}
By taking $\ep=\dfrac13$ in Lemma~\ref{lem3.2} and applying Lemma~\ref{lem3.1},
\begin{align*}
\norml{\na u_h}{\Om}^2\le &\frac13 k^2\norml{u_h}{\Om}^2+ \frac{C}{k h}\,k\norml{u_h}{\Ga}^2+\frac{C}{k^2 h^2\ga}\sum_{e\in\cE_h^I}\ga_e h_e\norml{\jm{\frac{\pa u_h}{\pa n_e}}}{e}^2\\
\le& \frac13\Big(2\norml{\na u_h}{\Om}^2 +\frac{C}{k^2}\norml{f}{\Om}^2+\frac{C}{k}\norml{g}{\Ga}^2\Big) \\
&+C\Big(\frac{1}{k h}+\frac{1}{ k^2 h^2\ga}\Big)\Big(\frac1k\norml{f}{\Om}\norml{\na u_h}{\Om}+\frac{1}{k^2}\norml{f}{\Om}^2+\frac{1}{k}\norml{g}{\Ga}^2\Big)\\
\le& \frac23\norml{\na u_h}{\Om}^2+\frac{C}{k}\Big(\frac{1}{k h}+\frac{1}{ k^2 h^2\ga}\Big)\norml{f}{\Om}\norml{\na u_h}{\Om} \\
&+C\Big(1+\frac{1}{k h}+\frac{1}{ k^2 h^2\ga}\Big)\Big(\frac{1}{k^2}\norml{f}{\Om}^2+\frac{1}{k}\norml{g}{\Ga}^2\Big)\\
\le& \frac56\norml{\na u_h}{\Om}^2+ C\Big(1+\frac{1}{k h}+\frac{1}{ k^2 h^2\ga}\Big)^2\frac{1}{k^2}\norml{f}{\Om}^2\\
&+C\Big(1+\frac{1}{k h}+\frac{1}{ k^2 h^2\ga}\Big)\frac{1}{k}\norml{g}{\Ga}^2.
\end{align*}
Therefore,
\begin{align*}
\norml{\na u_h}{\Om}^2
\lesssim \csta^2\norml{f}{\Om}^2+ \csta\norml{g}{\Ga}^2.
\end{align*}
Then the proof of the theorem follows by combining the above estimate and Lemma~\ref{lem3.1}.
\end{proof}
\begin{corollary}\label{cor}
The CIP-FEM \eqref{eicg} attains a unique solution for any $k>0$, $h>0$ and $\ga>0$.
\end{corollary}
\begin{remark} (a) For the general case when the penalty parameters or the meshes may be nonuniform, Theorem~\ref{tsta} and Corollary~\ref{cor} still hold with $\ga$ and $h$ replaced by $\gal=\min_{e\in\cE_h^I}\ga_e$ and $\hl=\min_{K\in\cT_h}{h_K}$, respectively. The proof is similar and is omitted. 

(b) If $\ga\gtrsim \dfrac{1}{k^3h^2}$ then $\csta\lesssim 1$ which implies the following stability estimates for the CIP-FEM:
\begin{align*}
\norm{u_h}_{1,h}\lesssim\norml{f}{\Om}+\norml{g}{\Ga}, \quad \norml{u_h}{\Om}\lesssim\frac{1}{k}\big(\norml{f}{\Om}+\norml{g}{\Ga}\big).
\end{align*}
These estimates are of the same order as those for the Helmholtz problem \eqref{e1.1}--\eqref{e1.2} (cf. Theorem~\ref{tstau}). But we do not suggest to choose $\ga$ as above when $k^3h^2$ is small, since a large $\ga$ may cause a large error of the discrete solution $u_h$ (cf. Theorem~\ref{tmain} below).

(c) The stability estimates in Theorem~\ref{tsta} will be improved in Section~\ref{sec-SEII} when $k^3h^2$ is small enough. Note that if $k^3h^2\gtrsim 1$ and $\ga\lesssim 1$, then $\csta\lesssim\dfrac{1}{\ga}$ and hence 
\begin{align*}
\norm{u_h}_{1,h}\lesssim\frac{1}{\ga}\big(\norml{f}{\Om}+\norml{g}{\Ga}\big), \quad \norml{u_h}{\Om}\lesssim\frac{1}{\ga k}\big(\norml{f}{\Om}+\norml{g}{\Ga}\big).
\end{align*}
\end{remark}
\section{Error estimates}\label{sec-err}
In this section we first introduce an elliptic projection of the solution $u$ to the Helmholtz problem \eqref{e1.1}--\eqref{e1.2} and estimate the error between them. Then we estimate the error between the elliptic projection and the CIP finite element solution $u_h$ by using the stability estimates in the previous section. In what follows, we assume that the domain $\Om$ is a convex polygon/polyhedron. Then $u\in H^2(\Om)$ (cf. \cite{g85}) and Theorem~\ref{tstau} implies that
\begin{equation}\label{estau2}
\norm{u}_{H^2(\Om)}\lesssim k\,M(f,g),
\end{equation}
where $M(f,g)=\norml{f}{\Om} + \norml{g}{\Ga}.$

\subsection{Elliptic projection and its error estimates}
For any $w\in E$, we define its
elliptic projection $\tilde{w}_h\in V_h$ by
\begin{equation}\label{e4.3}
a_h(\tilde{w}_h,v_h)+\i k\pd{\tilde{w}_h,v_h}  =a_h(w,v_h)+\i k\pd{w,v_h} 
\qquad\forall v_h\in V_h.
\end{equation}
In other words, $\tilde{w}_h$ is an CIP finite element approximation to
the solution $w$ of the following (complex-valued) Poisson problem:
\begin{alignat*}{2}
-\De w &= F &&\qquad \mbox{in } \Om,\\
\frac{\p w}{\p n} +\i k w &=\psi &&\qquad \mbox{on }\Gamma,
\end{alignat*}
for some given functions $F$ and $\psi$ which are determined by $w$.

The following lemma gives the continuity and coercivity of the sesquilinear form $a_h(\cdot,\cdot)$ whose proof is obvious and is omitted. 
\begin{lemma}\label{lem4.1}
For any $v, w\in E$, 
\begin{equation}
\abs{a_h(v,w)}, \;\abs{a_h(w,v)}\le \norm{v}_{1,h} \norm{w}_{1,h}, \label{e4.1}
\end{equation}
\begin{equation}
\re a_h(v,v)+\im a_h(v,v)= \norm{v}_{1,h}^2. \label{e4.2}
\end{equation}
\end{lemma}

Let $u$ be the solution of problem \eqref{e1.1}--\eqref{e1.2}
and $\tuh$ be its elliptic projection defined as above. Then \eqref{e4.3} immediately implies the following
Galerkin orthogonality:
\begin{equation}\label{e4.4}
a_h(u-\tuh,v_h)+\i k\pd{u-\tuh,v_h}  =0 \qquad\forall v_h\in V_h.
\end{equation}

\begin{lemma}\label{lem4.2}
There hold the
following estimates:
\begin{align}
\norm{u-\tuh}_{1,h} 
&\lesssim \big(1+\ga+kh\big)^{1/2} kh M(f,g), \label{e4.2a}\\
\norml{u-\tuh}{\Om} &\lesssim \big(1+\ga+kh\big) kh^2 M(f,g).  \label{e4.2b}
\end{align}
\end{lemma}

\begin{proof}
Let $\hat{u}_h\in V_h$ be the $P_1$-conforming finite element interpolant of $u$ on
the mesh $\cT_h$.  Then $\hat{u}_h$ satisfies the following
estimates (cf. \cite{bs08,ciarlet78}):
\begin{align}\label{e4.6}
\norml{u-\hat{u}_h}{\Om} \lesssim h^2 \abs{u}_{H^2(\Om)}, \qquad \norml{\na(u-\hat{u}_h)}{\Om} \lesssim h \abs{u}_{H^2(\Om)},
\end{align}
which imply that
\begin{align}
\norml{u-\hat{u}_h}{\Ga} &\lesssim h^{\frac32} \abs{u}_{H^2(\Om)},
\label{e4.6c}\\
\norm{u-\hat{u}_h}_{1,h} &\lesssim\big(1+\ga\big)^{1/2} \,h\abs{u}_{H^2(\Om)}, \label{e4.6b}
\end{align}
where we have used the trace inequality $\norml{w}{\Ga}\lesssim\norml{w}{\Om}\norm{w}_{H^1(\Om)}$ to derive \eqref{e4.6c} and used the local trace inequality $\norml{w}{\pa K}\lesssim h_K^{-1/2}\norml{w}{K}+h_K^{1/2}\norml{\na w}{K}$ for any $K\in\cT_h$ to derive \eqref{e4.6b}.

Let $\eta:=u-\tuh$. From \eqref{e4.4},
\begin{equation}\label{e4.5}
a_h(\eta,\eta)+\i k\pd{\eta,\eta} 
=a_h(\eta,u-\hat{u}_h)+\i k\pd{\eta,u-\hat{u}_h} .
\end{equation}
 It follows from Lemma~\ref{lem4.1} and \eqref{e4.5} that
\begin{align*}
\norm{\eta}_{1,h}^2=& \re a_h(\eta,\eta)
+\im a_h(\eta,\eta)\\
=&\re\bigl(a_h(\eta,\eta)+\i k\pd{\eta,\eta}  \bigr)\\
& +\im\left(a_h(\eta,\eta)+\i k\pd{\eta,\eta}  \right)
-k\pd{\eta,\eta}  \\
=&\re\bigl(a_h(\eta,u-\hat{u}_h)+\i k\pd{\eta,u-\hat{u}_h}  \bigr)
-k\norml{\eta}{\Ga}^2 \\
& +
\im\left(a_h(\eta,u-\hat{u}_h)+\i k\pd{\eta,u-\hat{u}_h}  \right)\\
\le &C \Bigl(\norm{\eta}_{1,h}\norm{u-\hat{u}_h}_{1,h}
+k\norml{\eta}{\Ga}\norml{u-\hat{u}_h}{\Ga}\Bigr)-k\norml{\eta}{\Ga}^2.
\end{align*}
Therefore, it follows from \eqref{e4.6c}, \eqref{e4.6b}, and \eqref{estau2} that
\begin{align}\label{e4.6a}
\norm{\eta}_{1,h}^2 + k\norml{\eta}{\Ga}^2
\lesssim &\norm{u-\hat{u}_h}_{1,h}^2
+k\norml{u-\hat{u}_h}{\Ga}^2\\
\lesssim& \big(1+\ga+kh\big)k^2h^2M(f,g)^2.\nn
\end{align}
That is, \eqref{e4.2a} holds.

To show \eqref{e4.2b}, we use the Nitsche's duality argument (cf. \cite{bs08,ciarlet78}).
Consider the following auxiliary problem:
\begin{alignat}{2}\label{e4.7}
-\De w &=u-\tuh &&\qquad\text{in }\Om,\\
\frac{\pa w}{\pa n}-\i k w &=0 &&\qquad\text{on }\Ga. \nn
\end{alignat}
It can be shown that $w$ satisfies
\begin{equation}\label{e4.8}
\abs{w}_{H^2(\Om)}\lesssim \norml{u-\tuh}{\Om}.
\end{equation}
Let $\hat{w}_h$ be the $P_1$-conforming finite element interpolant of $w$ on
$\cT_h$. Testing the conjugated of \eqref{e4.7} by $u-\tuh$ and using \eqref{e4.4} we get
\begin{align*}
\norml{u-\tuh}{\Om}^2 &=-(u-\tuh, \De w)=a_h(u-\tuh,w)+\i k\pd{u-\tuh,w}  \\
&=a_h(u-\tuh,w-\hat{w}_h)+\i k\pd{u-\tuh,w-\hat{w}_h}  \\
&\lesssim \norm{u-\tuh}_{1,h}\norm{w-\hat{w}_h}_{1,h}
+k\norml{u-\tuh}{\Ga}\norml{w-\hat{w}_h}{\Ga} \\
&\lesssim \norm{\eta}_{1,h} \big(1+\ga\big)^{1/2} h \abs{w}_{H^2(\Om)}
+k\norml{\eta}{\Ga} h^{\frac32}\abs{w}_{H^2(\Om)},
\end{align*}
which together with \eqref{e4.6a} and \eqref{e4.8} gives \eqref{e4.2b}.
The proof is completed.
\end{proof}

\subsection{Error estimates for the CIP-FEMs}
In this subsection we shall derive error estimates for the scheme \eqref{eicg}.
This will be done by exploiting the linearity of the Helmholtz equation
and making use of the stability estimates derived in Theorem \ref{tsta} and
the projection error estimates established in Lemma \ref{lem4.2}.

Let $u$ and $u_h$ denote the solution of  \eqref{e1.1}--\eqref{e1.2} and that of \eqref{eicg},
respectively. Define the error function $e_h:=u-u_h$. Subtracting \eqref{eicg}
from \eqref{2.6} with $v=v_h\in V_h$ yields the following error equation:
\begin{equation}\label{error-eq}
a_h(e_h,v_h) -k^2 (e_h,v_h) +\i k \pd{e_h,v_h}  =0
\qquad \forall v_h\in V_h.
\end{equation}
Let $\tuh$ be the elliptic projection of $u$ as defined in the previous
subsection. Write $e_h=\eta-\xi$ with
\[
\qquad \eta:=u-\tuh, \qquad \xi:=u_h-\tuh.
\]
From \eqref{error-eq} and \eqref{e4.4} we get
\begin{align}\label{e4.14}
a_h(\xi,v_h) -k^2 (\xi,v_h) +\i k \pd{\xi,v_h} 
&= a_h(\eta,v_h) -k^2 (\eta,v_h) +\i k \pd{ \eta,v_h} \\
&=-k^2 (\eta,v_h) \qquad \forall v_h\in V_h. \nn
\end{align}
The above equation implies that $\xi\in V_h$ is the solution of the
scheme \eqref{eicg} with source terms $f=-k^2\eta$ and $g\equiv 0$.
Then an application of Theorem \ref{tsta} and Lemma \ref{lem4.2}
immediately gives the following lemma.

\begin{lemma}\label{lem4.3}
$\xi=u_h-\tuh$ satisfies the following estimate:
\begin{align}\label{est_xi}
&\norm{\xi}_{1,h}+k \norml{\xi}{\Om}
\lesssim \csta\big(1+\ga+kh\big) k^3h^2M(f,g),
\end{align}
where $\csta$ is defined in Theorem~\ref{tsta}.
\end{lemma}

We are ready to state our error estimate results for scheme \eqref{eicg},
which follows from Lemma \ref{lem4.2},  Lemma \ref{lem4.3} and an
application of the triangle inequality.

\begin{theorem}\label{tmain}
Let $u$ and $u_h$ denote the solutions of \eqref{e1.1}--\eqref{e1.2}
and \eqref{eicg}, respectively.  
Then there exist two positive constants $C_1$ and $C_2$ such that the following error estimates hold.
\begin{align}\label{e4.15}
\norm{u-u_h}_{1,h} &\le\big(1+\ga+kh\big)\big(C_1 kh+C_2\csta k^3h^2\big)M(f,g),  \\
\norml{u-u_h}{\Om} &\le \big(1+\ga+kh\big) \big(C_1 kh^2+C_2\csta k^2h^2\big)M(f,g),\label{e4.16}
\end{align}
where $\csta$ is defined in Theorem~\ref{tsta} and $M(f,g)=\norml{f}{\Om} + \norml{g}{\Ga}$.
\end{theorem}

\begin{remark}\label{rtmain} (a) If $kh\lesssim 1, k^3h^2\gtrsim 1$, and $\ga\lesssim 1$,  then $\csta\lesssim \dfrac{1}{ k^3 h^2\ga}$ and we have the following error estimates for the CIP-FEM:
\[\norm{u-u_h}_{1,h} \lesssim\big(C_1 kh+\frac{C_2}{\ga}\big)M(f,g),  \quad
\norml{u-u_h}{\Om} \lesssim\big(C_1 kh^2+\frac{C_2}{k\ga}\big)M(f,g).\]
The pollution term in the above $H^1$ error estimate is $O(1)$ if $\ga\simeq 1$. By contrast the pollution term for the linear FEM when $ k^3h^2\gtrsim 1$ is expected to be of order $k^3h^2$ as that proved for the one dimensional case (cf. \cite{ib95a}). 

(b) The error estimates in Theorem~\ref{tmain} will be improved in the next section when $k^3h^2\le C_0$ for some constant $C_0$ independent of $k$, $h$, and the penalty parameters.
\end{remark}

\section{Stability-error iterative improvement}\label{sec-SEII}

In this section we improve the stability estimates in Theorem~\ref{tsta} and the error estimates in Theorem~\ref{tmain} under the condition that $k^3h^2$ is small enough, by using the trick of so called ``stability-error iterative improvement" developed in \cite{fw11}.

\begin{theorem}\label{tmain2}
Let $u$ and $u_h$ denote the solutions of \eqref{e1.1}--\eqref{e1.2} and \eqref{eicg},
respectively.  Assume that $\ga\lesssim 1$. Then there exists
a constant $C_0>0$, which is independent of $k$, $h$, and the penalty parameters,
such that if $k^3 h^2\le  C_0$, then the following stability and error
estimates hold:
\begin{align}\label{esta1}
\norm{u_h}_{1,h}&\lesssim  M(f,g),\\
\norml{u_h}{\Om}&\lesssim  \frac{1}{k} M(f,g), \label{esta1a}\\
\norm{u-u_h}_{1,h} &\lesssim\big(C_1 kh+C_2 k^3h^2\big)M(f,g), \label{eerr1} \\
\norml{u-u_h}{\Om} & \lesssim \big(C_1kh^2+C_2k^2h^2\big) M(f,g),
\label{eerr1a}
\end{align}
where $M(f,g)=\norml{f}{\Om} + \norml{g}{\Ga}$.
\end{theorem}
\begin{proof} It suffices to prove \eqref{esta1}, since \eqref{esta1a} follows then from Lemma~\ref{lem3.1} (specifically, \eqref{e3.1}) and \eqref{eerr1}--\eqref{eerr1a} follow from the improved stability
estimates and the argument used in the proof of Theorem~\ref{tmain}. Suppose $\csta>1$, otherwise, \eqref{esta1} holds already (cf. Theorem~\ref{tsta}).

From Theorem~\ref{tsta} we have, for any $f\in L^2(\Om)$ and $g\in L^2(\Ga)$,
\begin{equation}\label{e5.1}
\norm{u_h}_{1,h}\lesssim\csta M(f,g).
\end{equation}
Suppose $k h\lesssim 1$. Then \eqref{e4.14} and Lemma~\ref{lem4.2} imply that
\begin{align*}
\norm{u_h-\tuh}_{1,h}\lesssim\csta k^2\norml{u-\tuh}{\Om}
\lesssim \csta k^3h^2M(f,g).
\end{align*}
Therefore from the triangle inequality and Lemma~\ref{lem4.2} we have
\begin{equation}\label{e5.2}
    \norm{u-u_h}_{1,h}\lesssim  \norm{u-\tuh}_{1,h}+\norm{\tuh-u_h}_{1,h} \lesssim\big(kh+\csta k^3h^2\big)M(f,g).
\end{equation}
Now it follows from the triangle inequality and Theorem~\ref{tstau} that
\begin{align}\label{e5.3}
\norm{u_h}_{1,h} &\le\norm{u}_{1,h}+\norm{u_h-u}_{1,h}
=\norml{\na u}{\Om}+\norm{u-u_h}_{1,h}\\
&\lesssim \Bigl(1+k\,h+ \csta k^3 h^2 \Bigr) M(f,g).\nn
\end{align}
Repeating the above process yields that there exists a constant $\tilde C$
independent of $k$, $h$, and the penalty parameters, and a sequence of
positive numbers $\La_j$ such that
\begin{align}\label{e5.4}
\norm{u_h}_{1,h}\le \La_jM(f,g),
\end{align}
with
\begin{align*}
\La_0\simeq \csta ,\quad \La_j
=\tilde C(1+k\,h) +  \tilde C\,k^3 h^2\,\La_{j-1},
\quad j=1,2,\cdots.
\end{align*}
A simple calculation yields that if  $\tilde C\,k^3 h^2<\ta$ for
some positive constant $\ta<1$ then
\[
\lim_{j\to\infty}\La_j=\frac{\tilde C(1+k\,h) }{1-\tilde C\,k^3 h^2},
\]
which implies \eqref{esta1}.
\end{proof}
\begin{remark} (a)
Note that the stability estimates in \eqref{esta1} and \eqref{esta1a} are of the same order as the PDE stability estimates given
in Theorem~\ref{tstau}.

(b) Note that the estimates in Theorem~\ref{tmain2} are uniform with respect to $0<\ga\lesssim 1$. In the next section, by passing to the limit $\ga\to 0+$ in the CIP-FEMs \eqref{eicg} and in \eqref{esta1}--\eqref{eerr1a}, we will give stability and error estimates for the FEMs.
\end{remark}


\section{Stability and error estimates for the linear finite element method}\label{sec-fem}
It is clear that both the bilinear form $a_h(\cdot,\cdot)$ and the CIP finite element solution $u_h$ to \eqref{eicg} depend on the penalty parameters $\ga_e$. In this section, we choose $\ga_e\equiv\ga$ and denote by $a_h^\ga(\cdot,\cdot):=a_h(\cdot,\cdot)$ and by $u_h^\ga:=u_h$. Obviously, if $\ga$ vanishes, the CIP-FEM \eqref{eicg} ``degenerates" to  the standard linear FEM: Find $\uhfem\in V_h$ such that 
\begin{equation}\label{efem}
(\na \uhfem,\na v_h) - k^2(\uhfem,v_h) +\i k \pd{ \uhfem,v_h}
=(f, v_h)+\pd{g, v_h}   \qquad\forall v_h\in V_h.
\end{equation}

Next we will derive stability and error estimates for the FEM by showing that $u_h^\ga$ converges as $\ga\to 0+$.
\begin{theorem}\label{tmain3}
 There exists
a constant $C_0>0$ independent of $k$ and $h$
such that if $k^3 h^2\le  C_0$, then \eqref{efem} attains a unique solution $\uhfem\in V_h$ which satisfies the following stability and error estimates:
\begin{align}\label{estafem}
\norml{\na \uhfem}{\Om}&\lesssim  M(f,g),\\
\norml{\uhfem}{\Om}&\lesssim  \frac{1}{k} M(f,g), \label{estafema}\\
\norml{\na (u-\uhfem)}{\Om} & \lesssim\big(C_1 kh+C_2 k^3h^2\big)M(f,g), \label{eerrfem} \\
\norml{u-\uhfem}{\Om} & \lesssim \big(C_1kh^2+C_2k^2h^2\big) M(f,g),
\label{eerrfema}
\end{align}
where $M(f,g)=\norml{f}{\Om} + \norml{g}{\Ga}$.
\end{theorem}
\begin{proof}Let $C_0$ be the constant defined in Theorem~\ref{tmain2}. Suppose $k^3 h^2\le  C_0$ and $h$ is fixed. Note that the space $V_h$ is finite dimensional and hence any two norms on $V_h$ are equivalent. It is clear that, if $u_h^\ga$ converges to some function $\uhfem$ in $H^1(\Om)$ as $\ga\to 0+$, then the existence of a finite element solution to \eqref{efem} and the estimates \eqref{estafem}--\eqref{eerrfema} for the finite element solution follow by letting $\ga\to 0+$ in the CIP-FEM \eqref{eicg} and in Theorem~\ref{tmain2}. Next we prove the convergence of $u_h^\ga$ by using the Cauchy's convergence test.

By letting $\ga=\ga_1$, $\ga_2$ in \eqref{eicg}, respectively, and taking the difference, we get
\begin{align*}
a_h^{\ga_1}(u_h^{\ga_1},v_h)-a_h^{\ga_2}(u_h^{\ga_2},v_h)-k^2(u_h^{\ga_1}-u_h^{\ga_2},v_h) +\i k \pd{u_h^{\ga_1}-u_h^{\ga_2},v_h}  =0, \quad\forall v_h\in V_h.
\end{align*}
Recall that 
\[a_h^{\ga}(\vp,v)=(\na\vp,\na v)+  \i\ga \sum_{e\in\cE_h^{I}}h_e \pdjj{\frac{\pa \vp}{\pa n_e}}{\frac{\pa v}{\pa n_e}}, \quad\forall\vp, v\in E.
\]
Clearly,  $u_h^{\ga_1}-u_h^{\ga_2}$ is the solution of the following discrete problem:
\begin{align*}
a_h^{\ga_1}(u_h^{\ga_1}-&u_h^{\ga_2},v_h)-k^2(u_h^{\ga_1}-u_h^{\ga_2},v_h) +\i k \pd{u_h^{\ga_1}-u_h^{\ga_2},v_h} \\ 
&=(\ga_1-\ga_2)(f_h, v_h):=-\i(\ga_1-\ga_2)\sum_{e\in\cE_h^{I}}h_e \pdjj{\frac{\pa u_h^{\ga_2}}{\pa n_e}}{\frac{\pa v_h}{\pa n_e}},\quad\forall v_h\in V_h.
\end{align*}
Therefore, from Theorem~\ref{tmain2},
\begin{equation}\label{etmain3a}
\norm{u_h^{\ga_1}-u_h^{\ga_2}}_{H^1(\Om)}\lesssim \abs{\ga_1-\ga_2}\norml{f_h}{\Om}.
\end{equation}
Since any two norms on $V_h$ are equivalent,  from Theorem~\ref{tmain2} we have
\[\norml{f_h}{\Om}\lesssim C(h)\norm{u_h^{\ga_2}}_{H^1(\Om)}\lesssim C(h) M(f,g),\]
where $C(h)$ is some constant which is dependent of $h$ but independent of $\ga_1$ and $\ga_2$. 
 By combining the above two estimates, we have
\[\norm{u_h^{\ga_1}-u_h^{\ga_2}}_{H^1(\Om)}\lesssim \abs{\ga_1-\ga_2} C(h) M(f,g).\]
Thus $u_h^\ga$ converges in $H^1(\Om)$ as $\ga\to 0+$. 

It remains to prove the uniqueness. Let $u_h^0$ be any solution to the FEM \eqref{efem}. By repeating the lines for deriving \eqref{etmain3a} with $\ga_2=0$, we obtain,
\[\norm{u_h^{\ga_1}-u_h^0}_{H^1(\Om)}\lesssim \ga_1\norml{f_h}{\Om},\]
where $f_h$ depends on $u_h^0$ but is independent of $\ga_1$. Therefore $\lim_{\ga_1\to 0}u_h^{\ga_1}=u_h^0$ which implies that $u_h^0=\uhfem$. This completes the proof of the theorem.
\end{proof}

\section{Nmerical examples}\label{sec-num}

Throughout this section, we consider the following two-dimensional
Helmholtz problem:
\begin{alignat}{2}
-\De u - k^2 u &=f:=\frac{\sin(kr)}{r}  &&\qquad\mbox{in  } \Om,\label{e7.1}\\
\frac{\pa u}{\pa n} +\i k u &=g &&\qquad\mbox{on } \Ga_R:=\pa\Om.\label{e7.2}
\end{alignat}
Here $\Om$ is the unit regular hexagon with center $(0,0)$ (cf. Figure~\ref{fgm})
and $g$ is so chosen that the exact solution is
\begin{equation}\label{e7.3}
u=\frac{\cos(kr)}{k}-\frac{\cos k+\i\sin k}{k\big(J_0(k)+\i J_1(k)\big)}J_0(kr)
\end{equation}
in polar coordinates, where $J_\nu(z)$ are Bessel functions of the first kind.
\begin{figure}[ht]
\centerline{\includegraphics[scale=0.7]{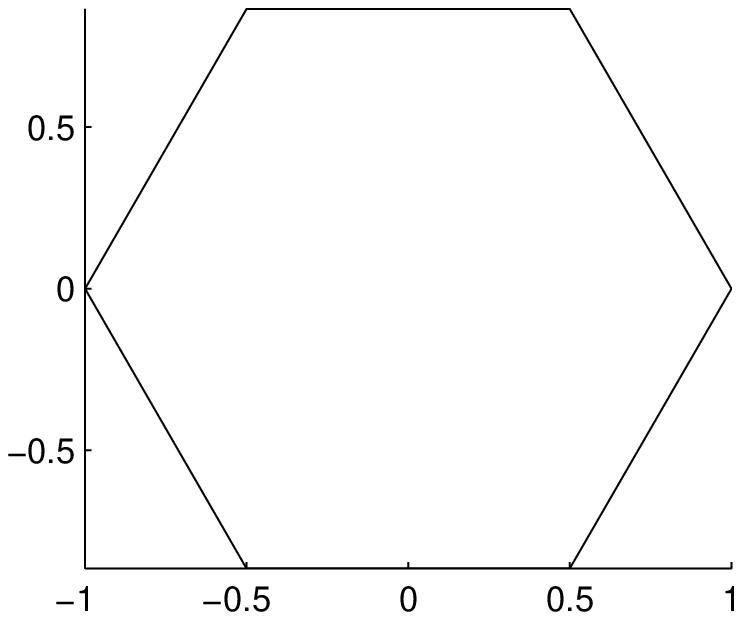} \qquad
\includegraphics[scale=0.7]{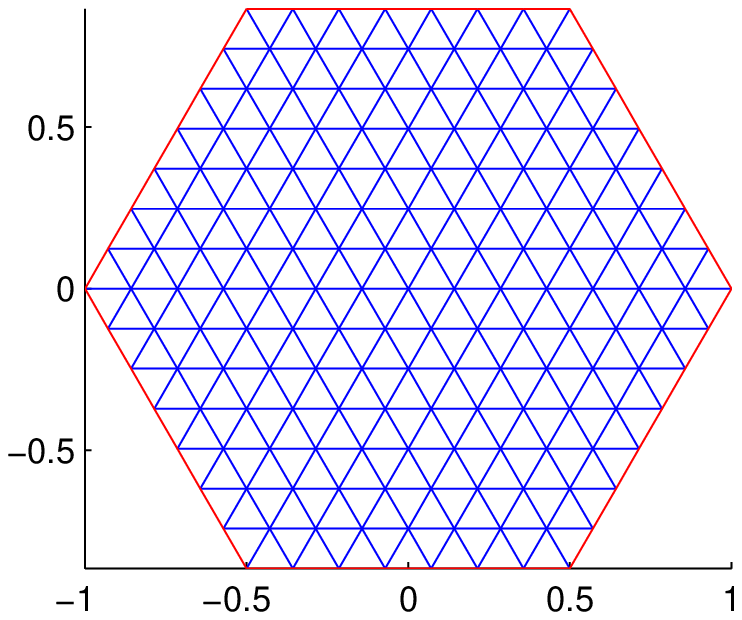}}
\caption{Geometry (left) and a sample mesh $\cT_{1/7}$ that consists of
congruent and equilateral triangles of size $h=1/7$ (right) for the
example.}\label{fgm}
\end{figure}

We remark that this problem has been computed in \cite{fw09} by using interior penalty discontinuous Galerkin methods. We use the same example for the convenience of comparison.

For any positive integer $m$, let $\cT_{1/m}$ denote the regular triangulation
that consists of $6m^2$ congruent and equilateral triangles of size $h=1/m$.
See Figure~\ref{fgm} (right) for a sample triangulation $\cT_{1/7}$. We remark that the number of total DOFs of the CIP-FEM on the triangulation $\cT_{1/m}$ is 
$3m^2+3m+1$ which is the same as that of the linear FEM and about one sixth of that of the linear IPDG method (cf. \cite{fw09}).

\subsection{Stability}\label{ssec-1} Given a triangulation $\cT_h$, recall that $u_h$ denotes the CIP finite element solution and $\uhfem$ denotes the $P_1$-conforming
finite element approximation of the problem \eqref{e7.1}--\eqref{e7.2}.
In this subsection, we use the following penalty parameters for the
CIP-FEM (cf. \eqref{eicg}):
\begin{equation}\label{e7.4}
\ga_e\equiv\ga=0.1 \quad\forall e\in\cE_h^I.
\end{equation}
Then,  according to
Theorem~\ref{tsta} and Theorem~\ref{tmain2}, we have the following stability estimate for the CIP finite element solution $u_h$.
\begin{equation}\label{e7.5}
\norml{\na u_h}{\Om}+k \norml{u_h}{\Om}\lesssim\min\big\{
 \norml{f}{\Om} + \norml{g}{\Ga},
\csta\norml{f}{\Om}+\csta^{1/2}\norml{g}{\Ga}\big\}
\end{equation}
where $\csta=\dfrac1k+\dfrac{1}{k^2 h}+\dfrac{10}{ k^3 h^2}.$
Noting that the stability estimate in $L^2$-norm is a direct consequence of that in $H^1$-seminorm (cf. Lemma~\ref{lem3.1}), we only examine the stability estimate for $\norml{\na u_h}{\Om}$. 

 Figure~\ref{fsta1} plots
the $H^1$-seminorm of the CIP finite element solution $\norml{\na u_h}{\Om}$, the $H^1$-seminorm
of the finite element solution
$\norml{\na \uhfem}{\Om}$ for $h=0.005$ and $0.002$, respectively,
and the $H^1$-seminorm of the exact solution $\norml{\na u}{\Om}$, for
$k=1, \cdots, 500$. It shows that
 $\norml{\na u}{\Om}\simeq 1$, $\norml{\na u_h}{\Om}\lesssim 1$ and $\norml{\na u_h}{\Om}$ decreases for $k$ large enough as indicated by \eqref{e7.5}. It is also shown that
$\norml{\na \uhfem}{\Om}\lesssim 1$ which is guaranteed theoretically only for $k^3h^2$ small enough (cf. Theroem~\ref{tmain3}).

\begin{figure}[ht]
\centerline{\includegraphics[scale=0.62]{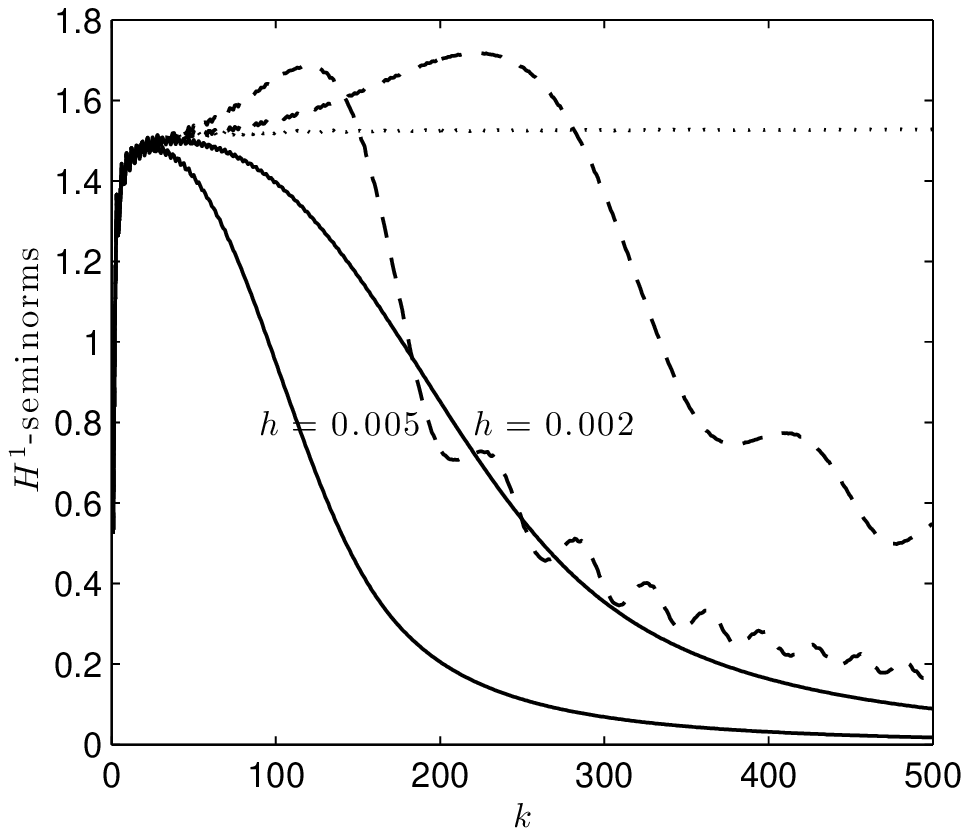}}
\caption{$\norml{\na u_h}{\Om}$ (solid), $\norml{\na \uhfem}{\Om}$  (dashed)
for $h=0.05$ and $0.005$, respectively. The dotted line gives the
$H^1$-seminorm of the exact solution $\norml{\na u}{\Om}$.}\label{fsta1}
\end{figure}

Figure~\ref{fsta2} shows stability behaviors of $u_h$ and $\uhfem$ when $kh=1$ for $k=1,2,\cdots, 500$, which satisfies the ``rule of thumb". Note that $\csta=12/k$ for $kh=1$.
It is shown that $\norml{\na u_h}{\Om}$ 
is in inverse proportion to $k$ for $k$ large which means the term $\csta\norml{f}{\Om}$ dominates the stability bound for $k\le 500$ (cf. \eqref{e7.5}). As a matter of fact, numerical integrations show that $\norml{f}{\Om}$ is more than $25$ times $\norml{g}{\Ga}$ for $k=200,\cdots,500$. 
\begin{figure}[ht]
\centerline{\includegraphics[scale=0.62]{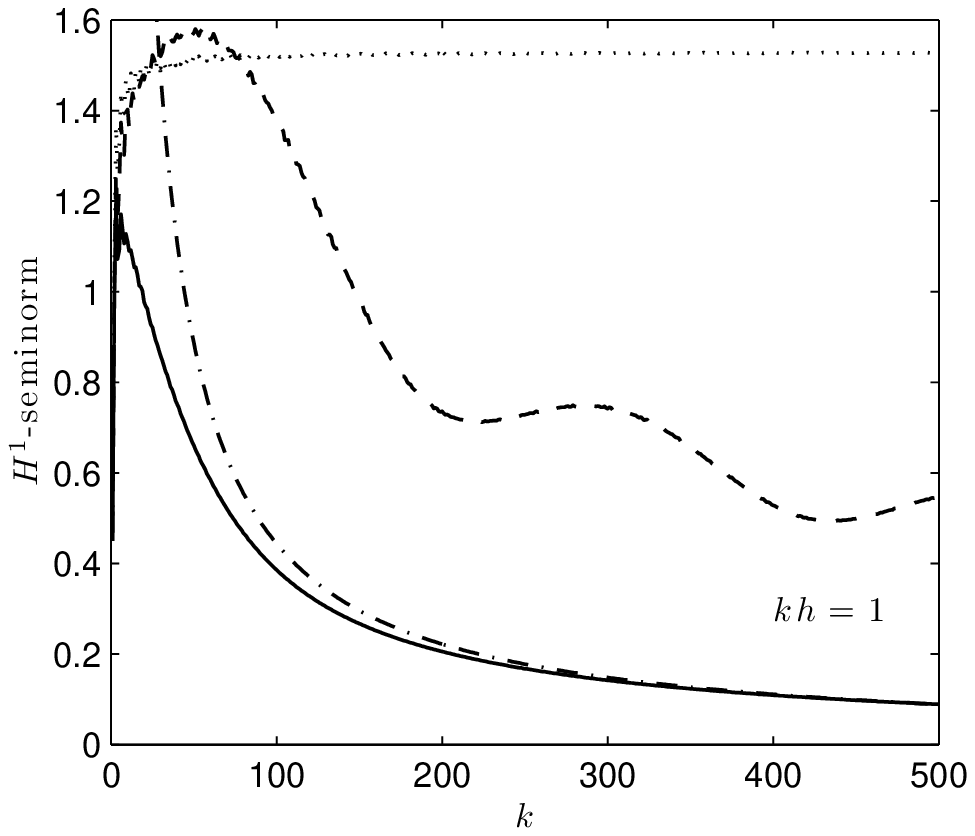}}
\caption{$\norml{\na u_h}{\Om}$ (solid) and $c/k$  (dash dotted) versus to $k$ with $kh=1$, where $c$ is some constant. The dashed line gives $\norml{\na\uhfem}{\Om}$. The dotted line gives the
$H^1$-seminorm of the exact solution $\norml{\na u}{\Om}$.}\label{fsta2}
\end{figure}

\subsection{Error estimates}\label{ssec-2} In this subsection, we use the same penalty parameter ($\ga=0.1$) for the CIP-FEM as in \eqref{e7.4}.
From Theorem~\ref{tmain} (cf. Remark~\ref{rtmain}) and Theorem~\ref{tmain2}, the error of the CIP finite element solution in the $H^1$-seminorm is bounded by
\begin{equation}\label{e7.6}
\norml{\na(u-u_h)}{\Om}\le C_1kh+C_2\min\set{k^3h^2,1}
\end{equation}
for some constants $C_1$ and $C_2$ if $kh\lesssim 1$. On the other hand, from Theorem~\ref{tmain3}, the error of the finite element solution in the $H^1$-seminorm is bounded by
\begin{equation}\label{e7.7}
\norml{\na(u-\uhfem)}{\Om}\le C_1kh+C_2k^3h^2
\end{equation}
for some constants $C_1$ and $C_2$ if $k^3 h^2\le  C_0$.
The second terms on the right hand
sides of \eqref{e7.6} and \eqref{e7.7} are the so-called pollution errors. We now present numerical results
to verify the above error bounds.

In the left graph of Figure~\ref{ferr1}, the relative error of the CIP finite element solution
with parameters given by \eqref{e7.4} and the relative error of the finite element
interpolant are displayed in one plot. When the mesh size is decreasing, the relative error of the
CIP finite element solution stays around $100\%$ before it is less than $100\%$,
then decays slowly on a range increasing with $k$, and then decays at a
rate greater than $-1$ in the log-log scale but converges as fast as the
finite element interpolant (with slope $-1$) for small $h$. The relative
error grows with $k$ along line $k h=0.25.$ By contrast, as shown in the right of Figure~\ref{ferr1}, the relative error of the finite element solution first
oscillates around $100\%$, then decays at a rate greater than $-1$ in
the log-log scale but converges as fast as the finite element interpolant
(with slope $-1$) for small $h$. The relative error of the finite element solution also grows with $k$ along line $k h=0.25. $ 
\begin{figure}[ht]
\centerline{
\includegraphics[scale=0.62]{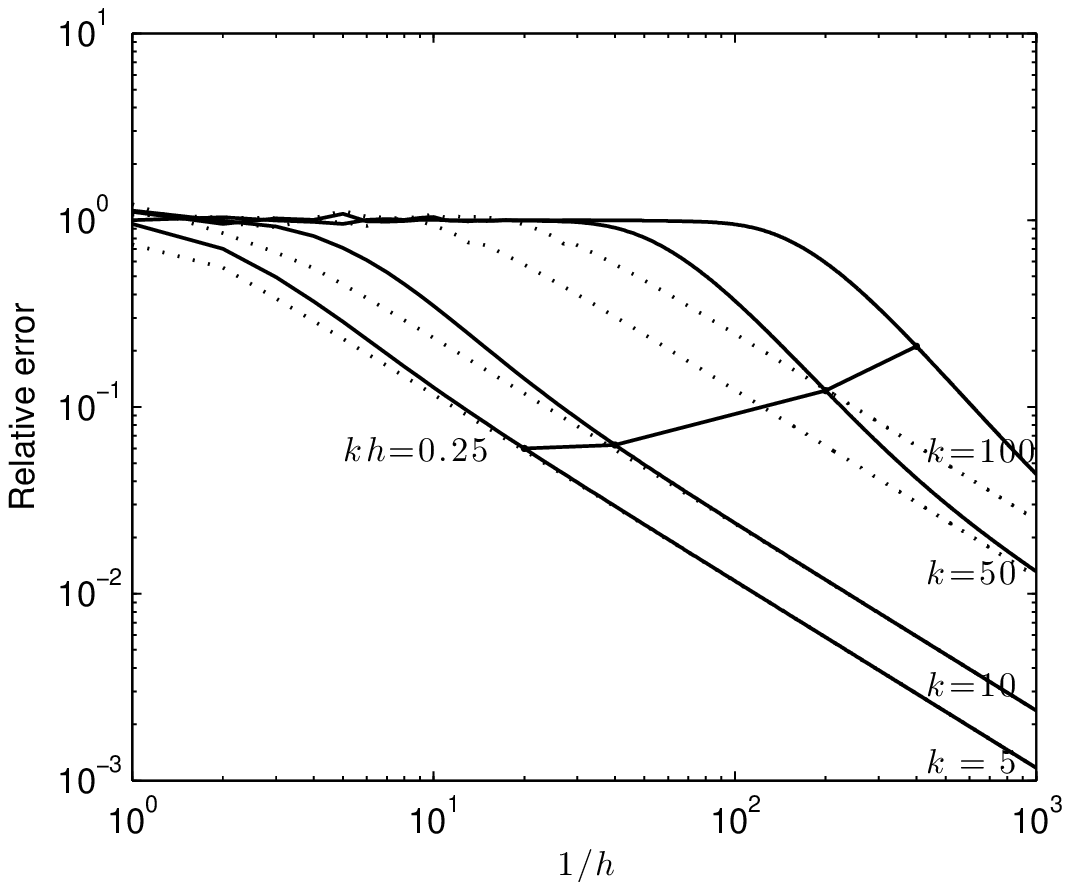}
\includegraphics[scale=0.62]{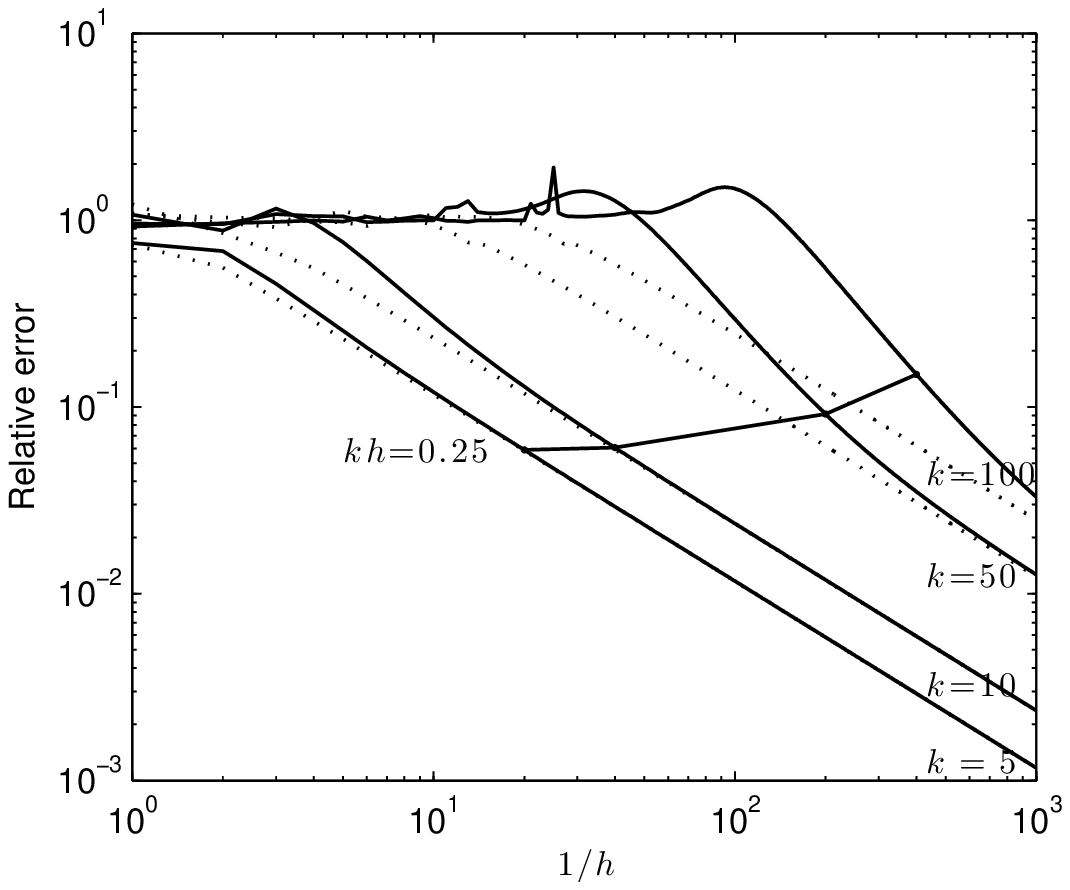}
}
\caption{Left graph: the relative error of the CIP finite element solution  with parameters
given by \eqref{e7.4} (solid) and the relative error of the finite element
interpolant (dotted) in $H^1$-seminorm for $k=5, k=10, k=50,$ and $k=100$,
respectively. Right graph: corresponding plots for finite element solutions.}\label{ferr1}
\end{figure}

Unlike the error of the finite
element interpolant, both the error of the CIP finite element solution and that of the finite element solution are not controlled by
the magnitude of $k h$ as indicated also by the two graphs in Figure~\ref{ferr2}. It is shown that when $h$ is determined according to the ``rule of thumb", the relative error of the CIP finite element solution keeps less than $100\%$  which means the CIP finite element solution has some accuracy even for large $k$, while the finite element solution is totally unusable for large $k$.  Figure~\ref{fsurf} displays the surface
plots of the real parts of the linear interpolant of the exact solution (left), the CIP finite element
solution with parameters given by \eqref{e7.4}  (center) , and the finite element solution (right), for $k=100$
on the mesh with mesh size $h=1/100$.
It is shown that the CIP finite element solution has a correct shape although
its amplitude is not very accurate. By contrast, the finite element solution has both wrong shape and amplitude. We remark that the accuracy of the CIP finite solution can be further greatly improved by tuning  the penalty parameter $\i\ga$, see Subsection~\ref{ssec-3} below.

\begin{figure}[ht]
\centerline{
\includegraphics[scale=0.62]{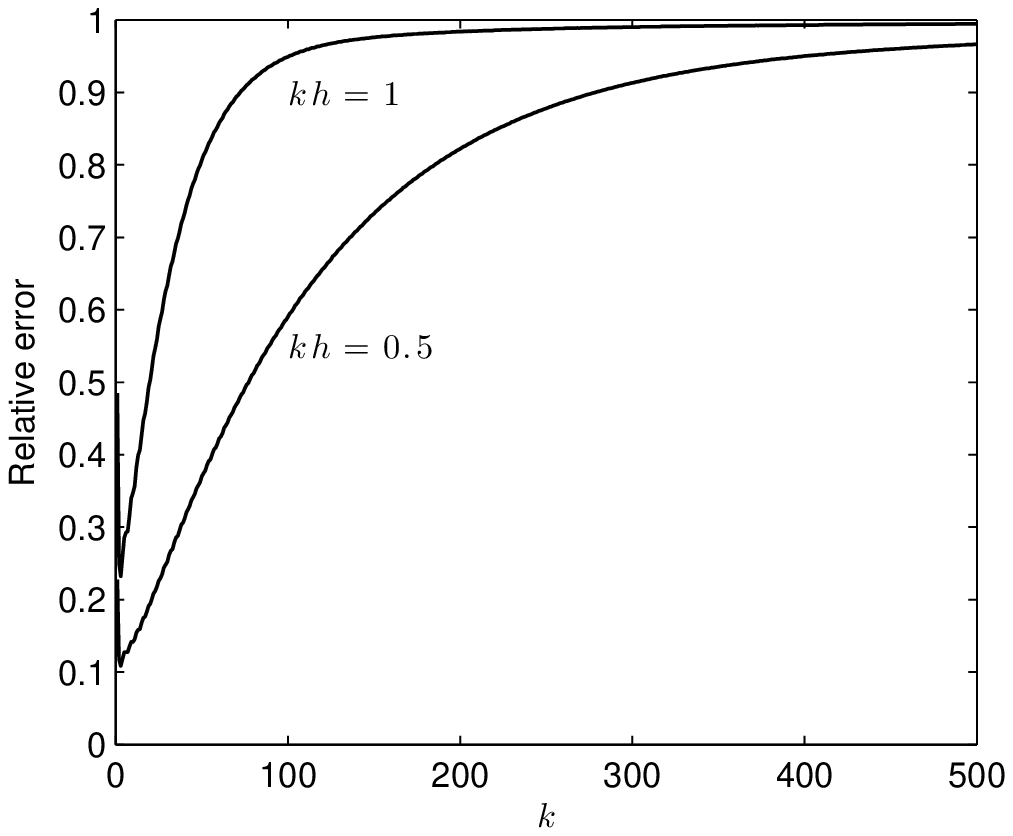}
\includegraphics[scale=0.62]{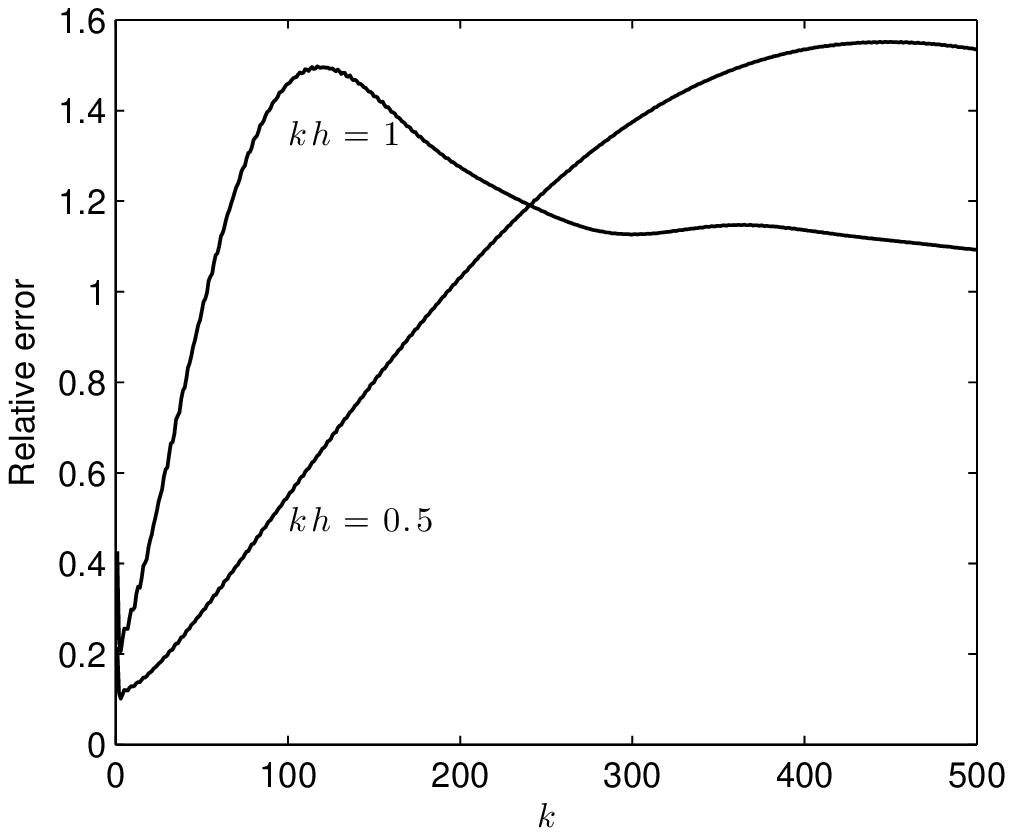}
}
\caption{The relative error of the CIP finite element solution (left) with parameters given
by \eqref{e7.4} and that of the finite element solution (right) in $H^1$-seminorm computed for $k=1, 2, \cdots, 500$ with mesh size $h$  determined by $kh=1$ and  $kh=0.5$, respectively.}\label{ferr2}
\end{figure}

\begin{figure}[ht]
\centerline{\includegraphics[scale=0.6]{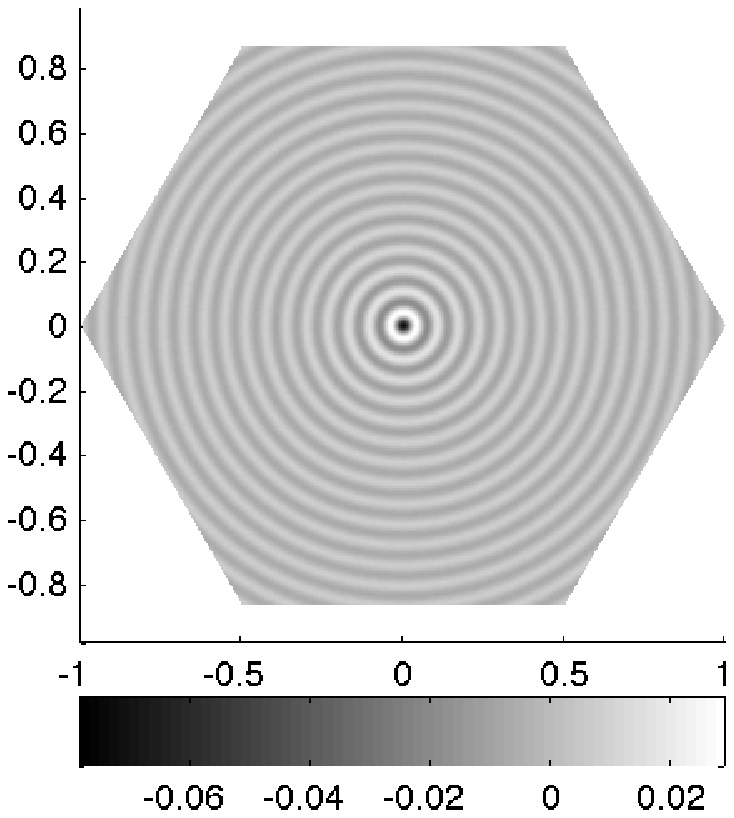}
\includegraphics[scale=0.6]{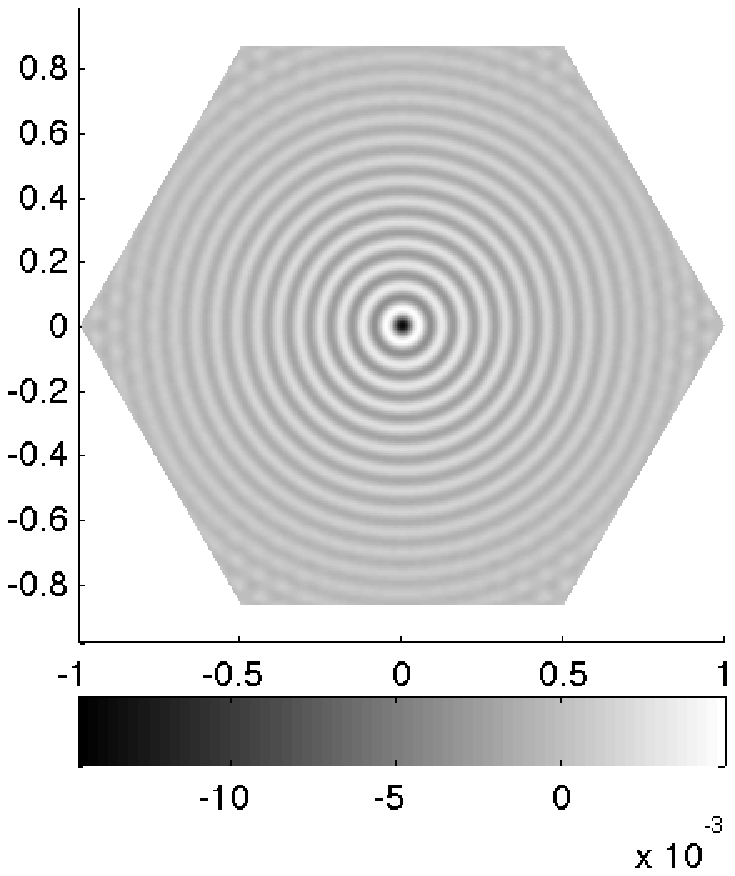}
\includegraphics[scale=0.6]{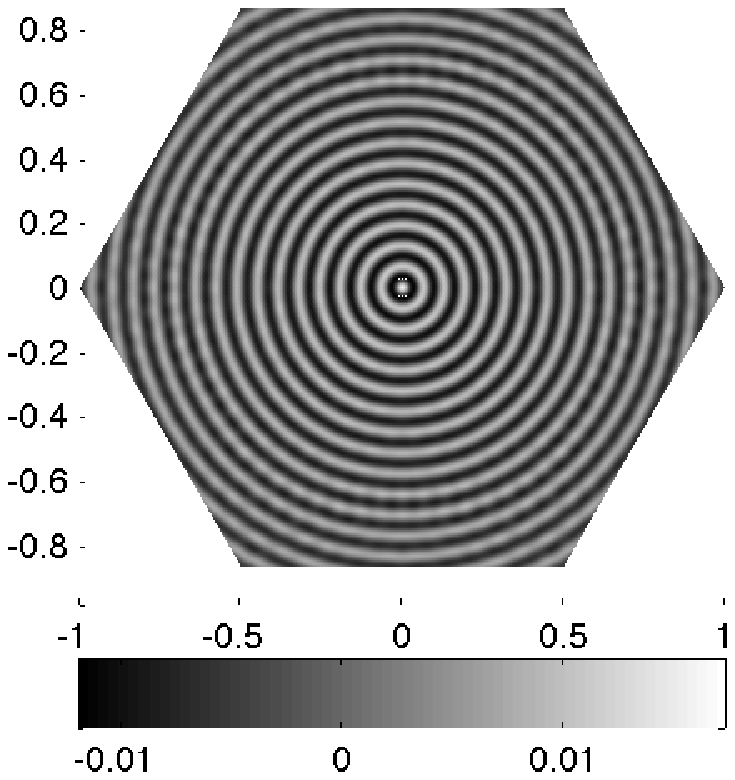}}
\caption{Surface plots of the real parts of the linear interpolant (left), the CIP finite element
solution with parameters given by \eqref{e7.4}  (center), and the finite element solution (right), for $k=100$
on the mesh with mesh size $h=1/100$.}\label{fsurf}
\end{figure}

Next we verify more precisely the pollution terms in \eqref{e7.6} and \eqref{e7.7}. To do so, we introduce the definition of the critical mesh size with respect to a given relative tolerance. 
\begin{definition}
Given a relative tolerance $\ep$ and a wave number $k$, the critical mesh size $h(k,\ep)$ with respect to the relative tolerance $\ep$  is defined by the maximum mesh size such that the relative error of the CIP finite element solution (or the finite element solution) in $H^1$-seminorm is less than or equal to $\ep$.
\end{definition}

It is clear that, if the pollution terms in \eqref{e7.6} and \eqref{e7.7} are of order $k^3h^2$, then $h(k,\ep)$ should be proportional to $k^{-3/2}$ for $k$ large enough. This is verified by Figure~\ref{ferr3} which plots $h(k,0.5)$ versus $k$ for the CIP finite element solution (left) with parameters given
by \eqref{e7.4} and for the finite element solution (right), respectively. We remark that the maximum wave number such that $h(k,0.5)\ge 0.001$ is $k_{\max}=266$ for the CIP-FEM with parameters given by \eqref{e7.4} and is $k_{\max}=280$ for the FEM. Note that if the mesh size $h=0.001$, then the number of total DOFs of the CIP finite element system is $3,003,001$, so is that of the FEM. Therefore, $k=k_{\max}$ is the maximum wave number such that the problem \eqref{e7.1}--\eqref{e7.2} can be approximated by the CIP-FEM (or FEM) with  relative error in $H^1$-seminorm $\le50\%$ while using at most $3,003,001$ total DOFs.
\begin{figure}[ht]
\centerline{
\includegraphics[scale=0.62]{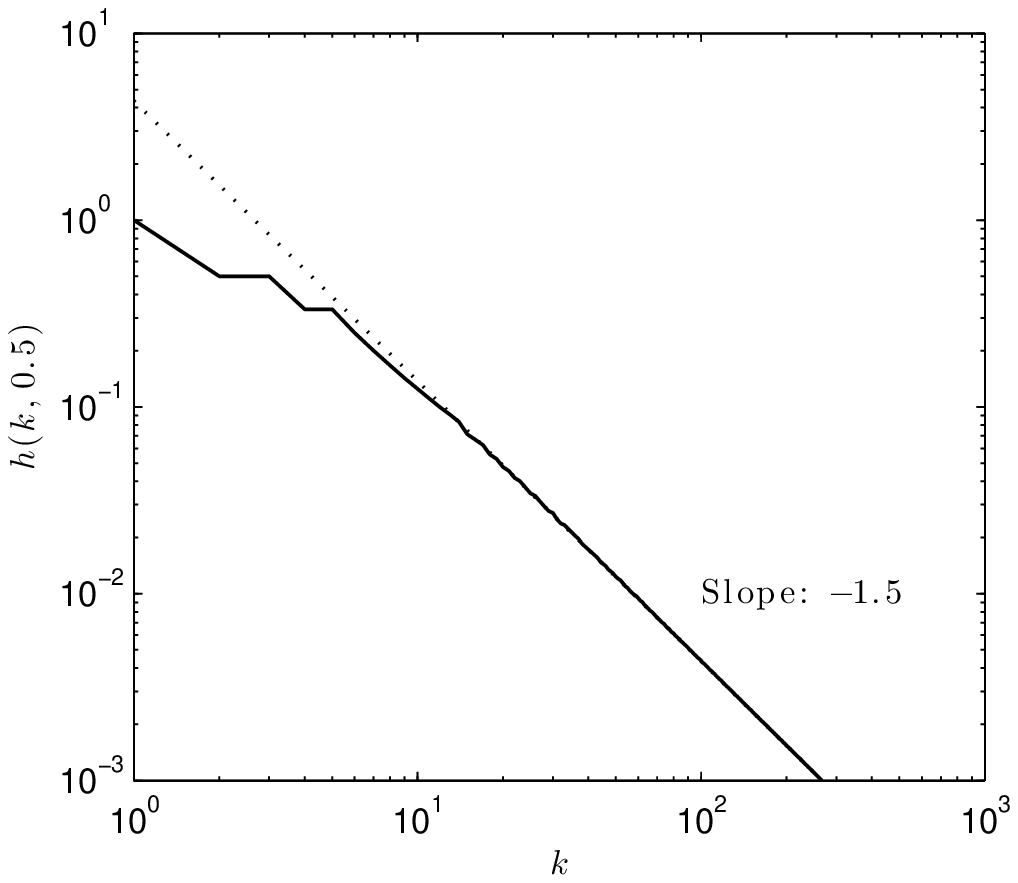}
\includegraphics[scale=0.62]{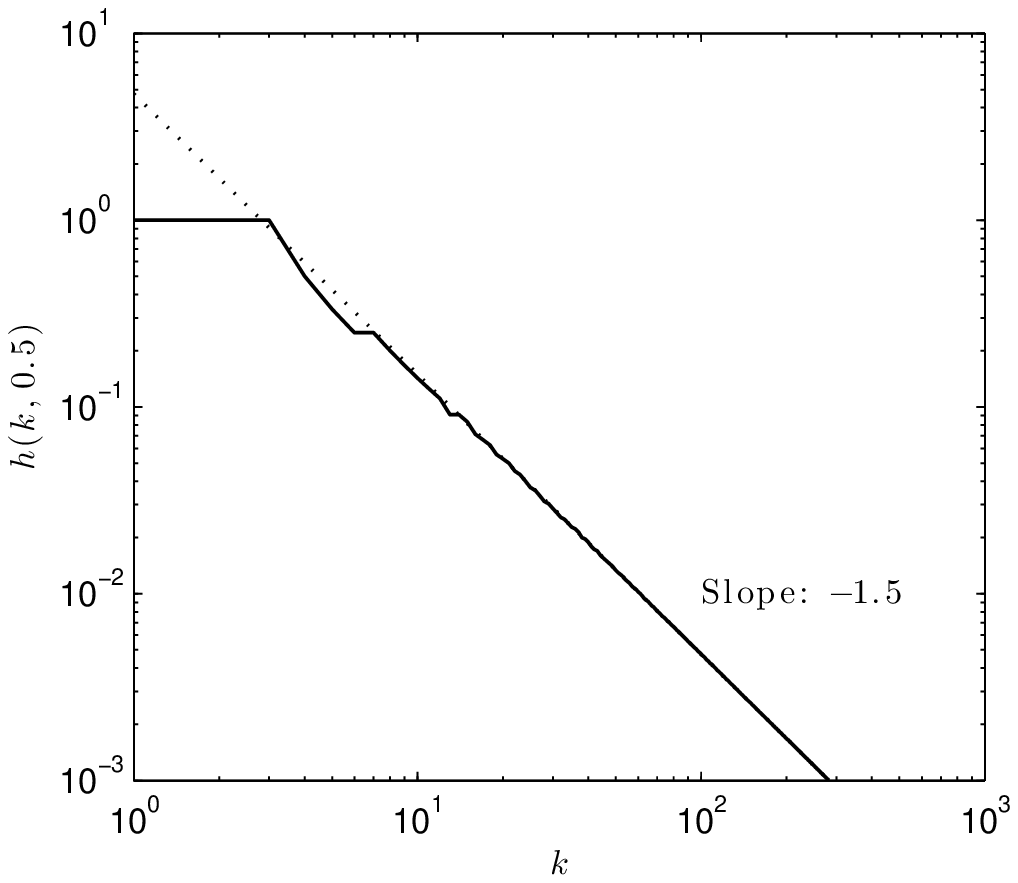}
}
\caption{$h(k,0.5)$ versus $k$ for the CIP-FEM (left) with parameters given
by \eqref{e7.4} and for the FEM (right), respectively. The dotted lines give lines of slope $-1.5$ in the log-log scale.}\label{ferr3}
\end{figure}

\subsection{Reduction of the pollution effect}\label{ssec-3}
 In \cite{fw09}, it is shown that appropriate choice of the penalty parameters
 can significantly reduce the pollution error of the symmetric IPDG method. 
 In this subsection, we shall show that the same thing holds true for the CIP-FEM. We use the following parameters:
\begin{equation}\label{e7.8}
    \i\ga_e\equiv\i\ga=-0.07+0.01\i.
\end{equation}
We remark that this choice of $\i\ga_e$ is the same as the choice of the penalty parameters $\i\ga_{1,e}$ from \cite[Subsection~6.4]{fw09} for the IPDG method.

The relative error of the CIP finite element solution with parameters given by \eqref{e7.8}
and the relative error of the finite element interpolant are displayed in the
left graph of Figure~\ref{ferr12b}. The CIP-FEM with  parameters given
by \eqref{e7.8} is much better than both the CIP-FEM using  parameters given
by \eqref{e7.4} and the FEM (cf. Figure~\ref{ferr1} and Figure~\ref{ferr2}). The relative error does not
increase significantly with the change of $k$ along line
$k h=0.25$ for $k\le 100$. But this does not mean that the pollution error
has been eliminated.
\begin{figure}[ht]
\centerline{
\includegraphics[scale=0.62]{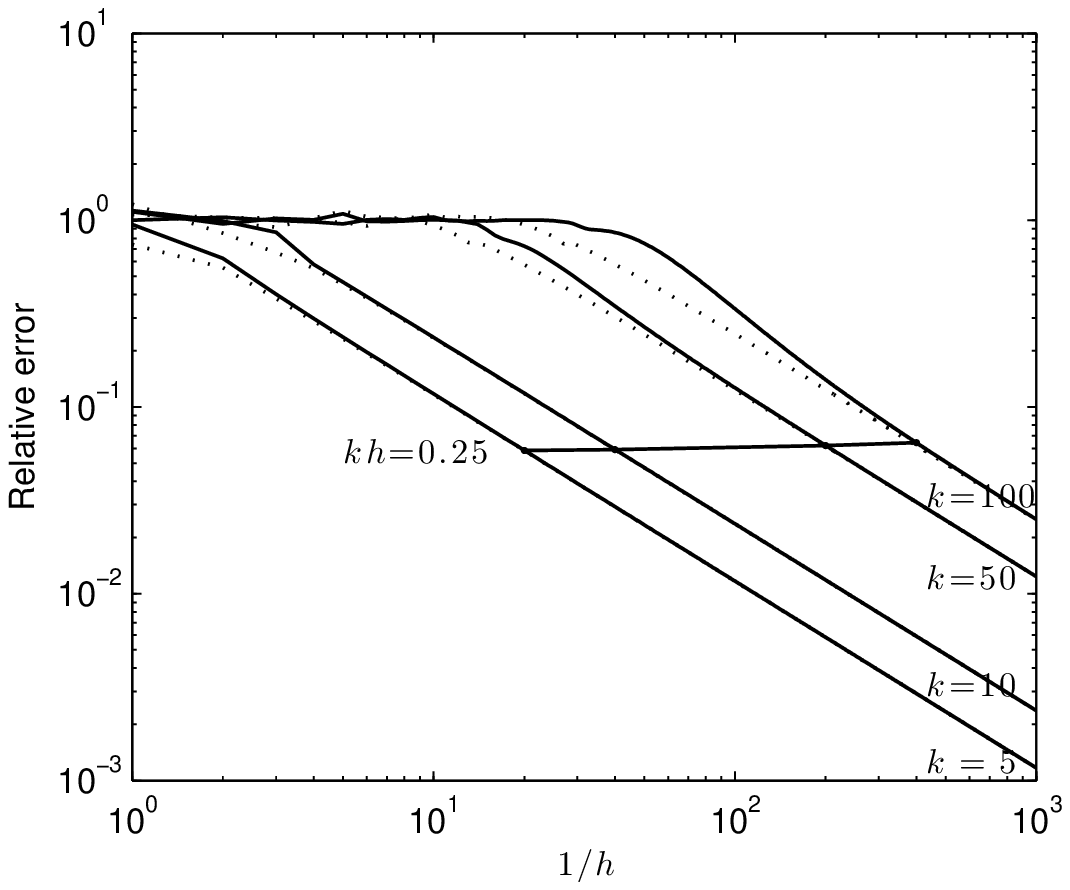}
\includegraphics[scale=0.62]{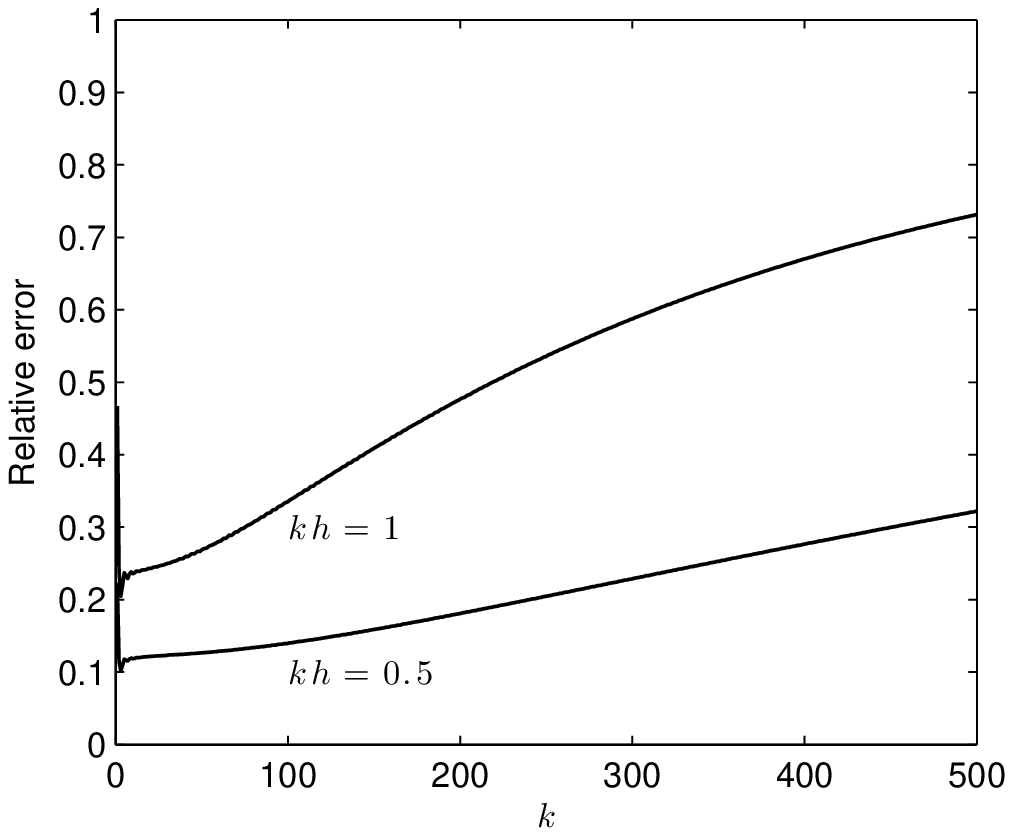}
}
\caption{Left graph: the relative error of the CIP finite element solution  with parameters given
by \eqref{e7.8} (solid) and the relative error of the finite element
interpolant (dotted) in $H^1$-seminorm for $k=5, k=10, k=50,$
and $k=100$, respectively. Right graph: the relative error of the CIP finite element solution with
parameters given by \eqref{e7.8} in $H^1$-seminorm computed for $k=1, 2, \cdots, 500$
with mesh size $h$ determined by $kh=1$ and $kh=0.5$, respectively.}\label{ferr12b}
\end{figure}
For more detailed observation, the relative errors of the CIP finite element solution
with parameters given by \eqref{e7.8}, computed for all
integer $k$ from $1$ to $500$  for $k h=1$ and $k h=0.5$, are plotted in
the right graph of Figure~\ref{ferr12b}. It is shown that the pollution error is
reduced significantly.

Figure~\ref{ferr3b} plots $h(k,0.5)$, the critical mesh size with respect to the relative tolerance $50\%$, versus $k$ for the CIP-FEM with parameters given
by \eqref{e7.8}. We recall that $h(k,0.5)$ is the maximum mesh size such that the relative error of the CIP finite element solution in $H^1$-seminorm is less than or equal to $50\%$. The decreasing rate of $h(k,0.5)$ in the log-log scale is less than $-1.5$ for $k$ from $1$ to a relatively large value, which means that the pollution effect is reduced.  
We remark that the maximum wave number under the condition $h(k,0.5)\ge 0.001$ is $k_{\max}=622$ for the CIP-FEM with parameters given by \eqref{e7.8} which is more than twice of that for the FEM.
\begin{figure}[ht]
\centerline{
\includegraphics[scale=0.62]{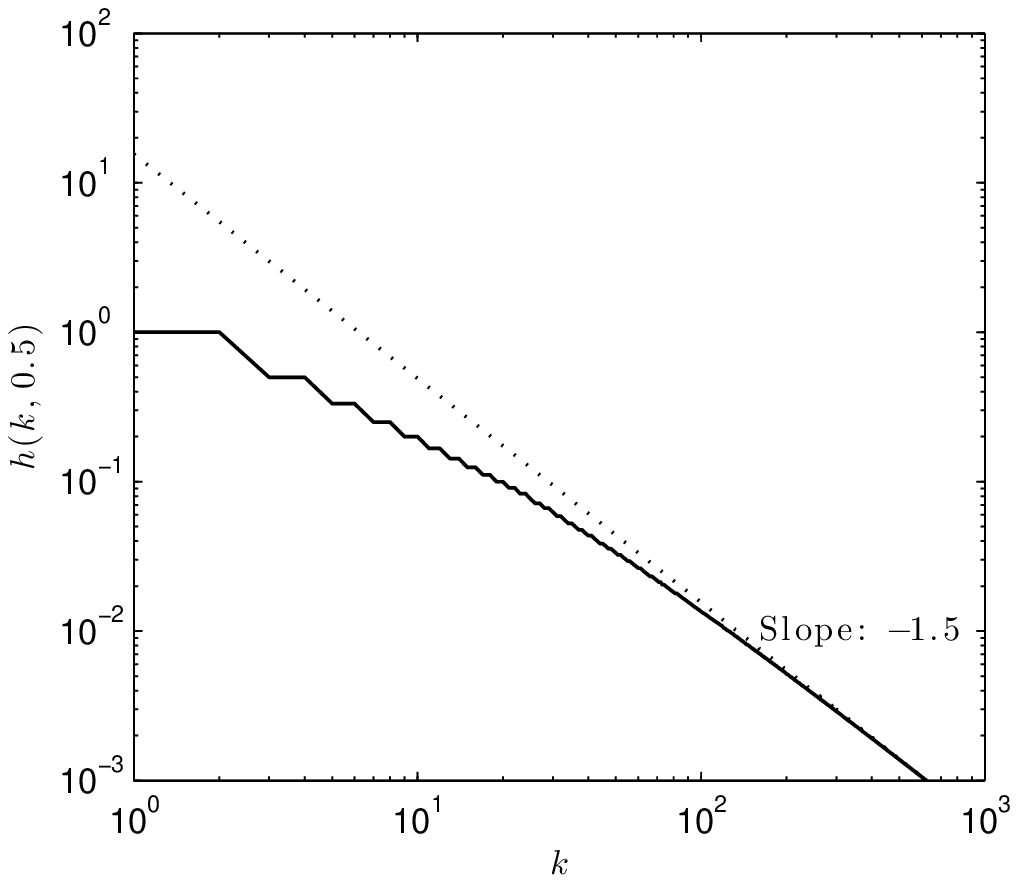}
}
\caption{$h(k,0.5)$ versus $k$ for the CIP-FEM with parameters given
by \eqref{e7.8}. The dotted line gives a line of slope $-1.5$ in the log-log scale}\label{ferr3b}
\end{figure}

 For more detailed comparison between the CIP-FEM and the FEM, we consider the problem
\eqref{e7.1}--\eqref{e7.2} with wave number $k=100$.
The traces of the CIP finite element solutions with parameters given by \eqref{e7.8}
and the finite element solutions in the $xz$-plane for mesh sizes
$h=1/50, 1/120$, and $1/200$, and the trace of the exact
solution in the $xz$-plane, are plotted in Figure~\ref{ftrace}. The shape
of the CIP finite element solution is roughly same as that of the exact solution
for $h=1/50$. They match very well for $h=1/120$ and even
better for $h=1/200$. While the finite element solution has a wrong
shape near the origin for $h=1/50$ and $h=1/120$ and only has a correct
shape for $h=1/200$. The phase error appears in all the three cases for
the finite element solution. We remark that the figures in the left of Figure~\ref{ftrace} look almost the same as those in the left of Figure~6.11 in \cite{fw09}, which means that the CIP-FEM has almost the same accuracy as the IPDG method analyzed in \cite{fw09} on the same mesh while using about one sixth of total DOFs of it.
\begin{figure}[ht]
\centerline{\includegraphics[scale=0.6]{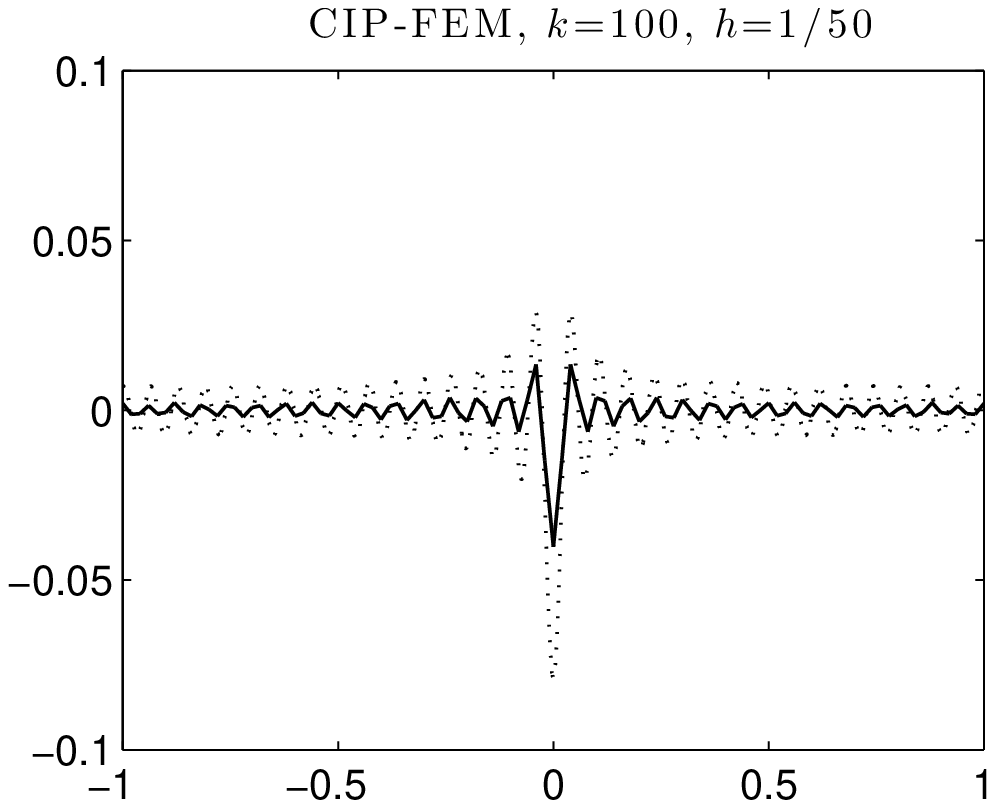}
\includegraphics[scale=0.6]{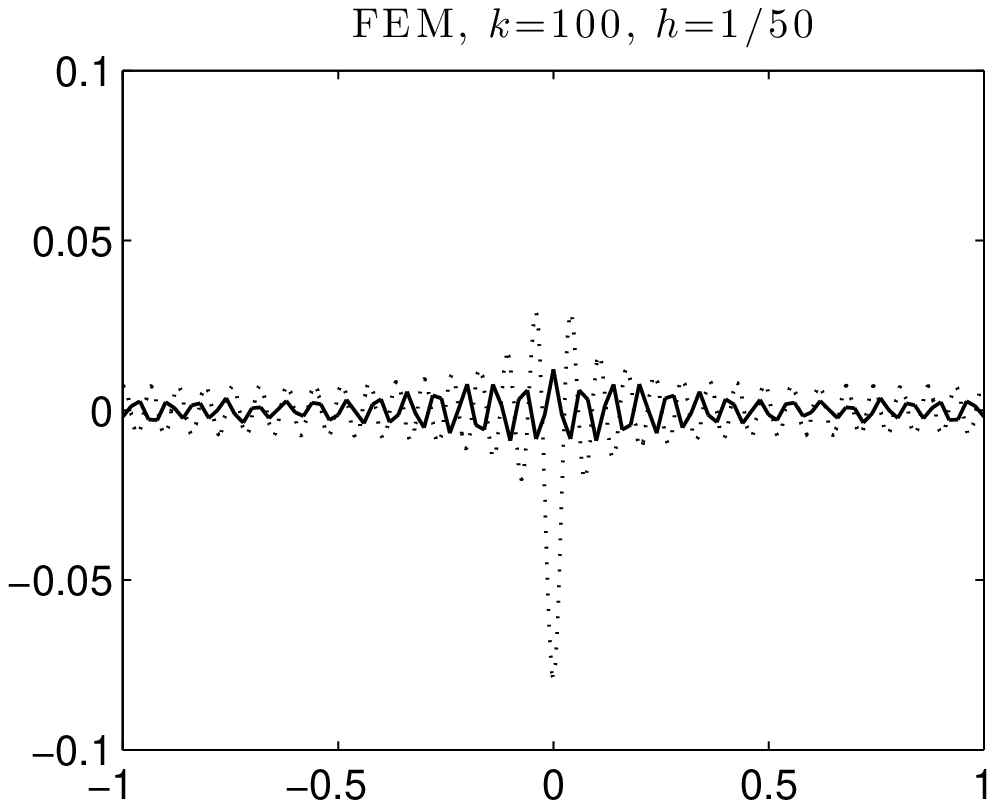}}
\centerline{\includegraphics[scale=0.6]{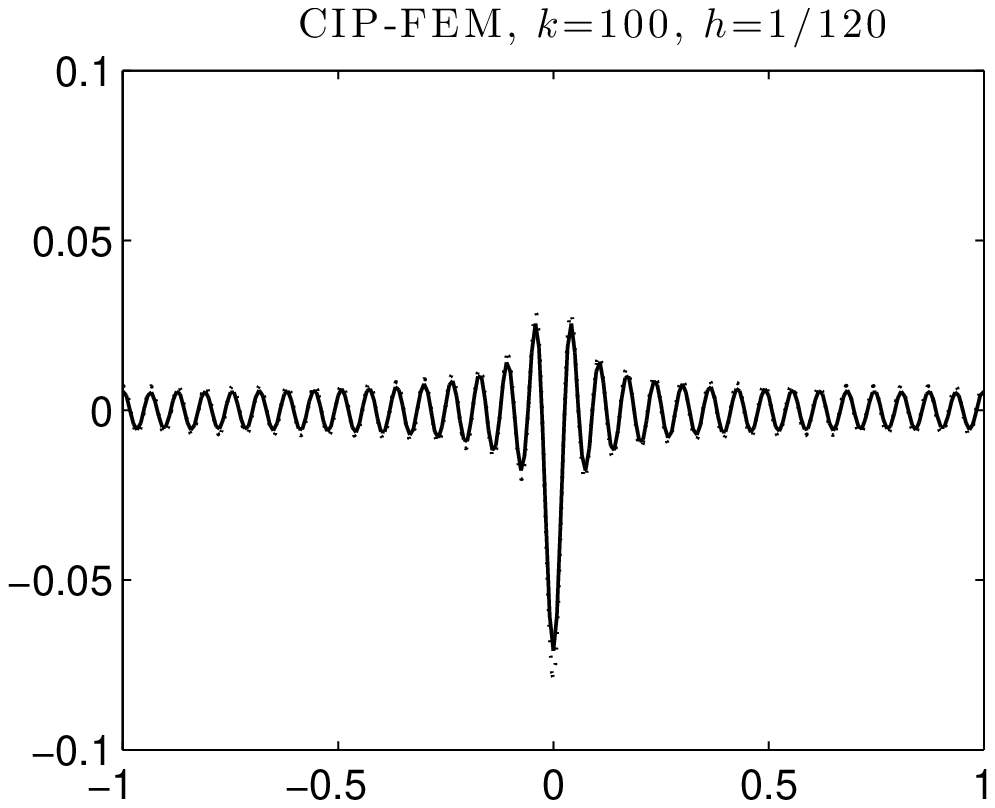}
\includegraphics[scale=0.6]{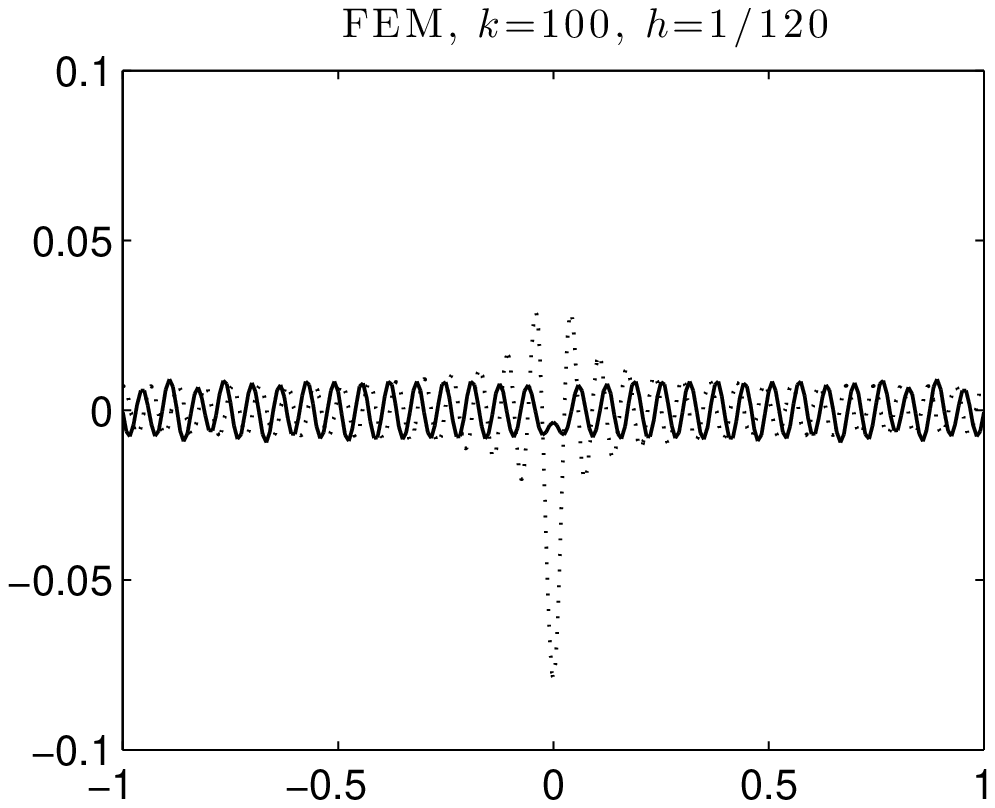}}
\centerline{\includegraphics[scale=0.6]{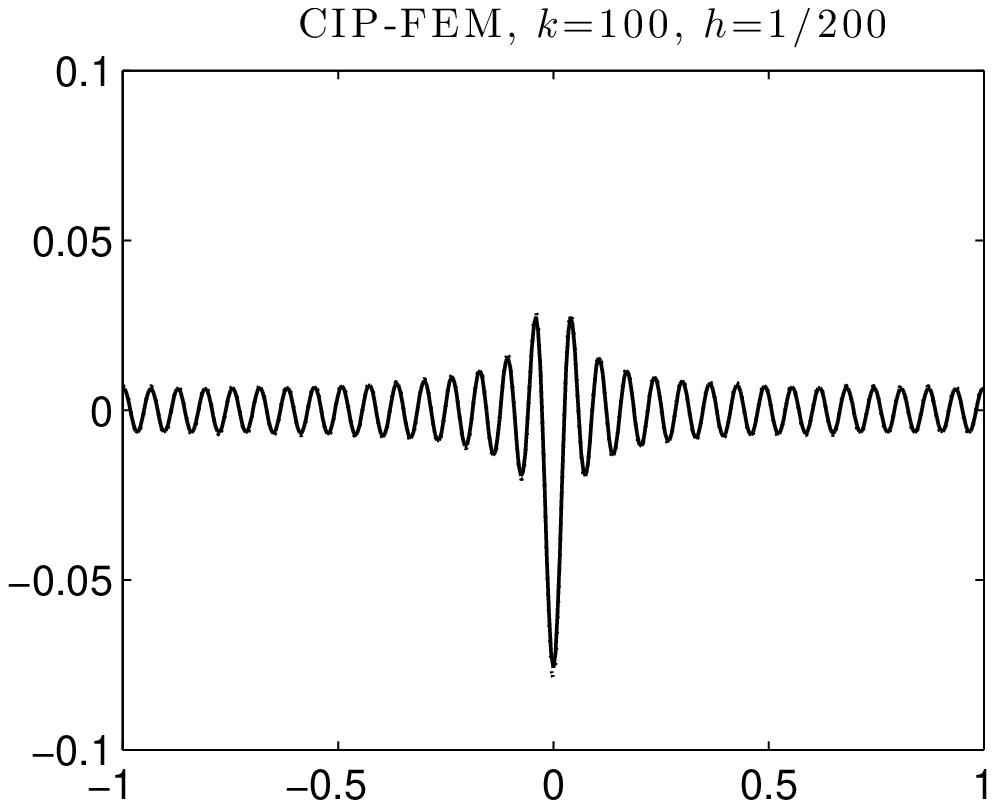}
\includegraphics[scale=0.6]{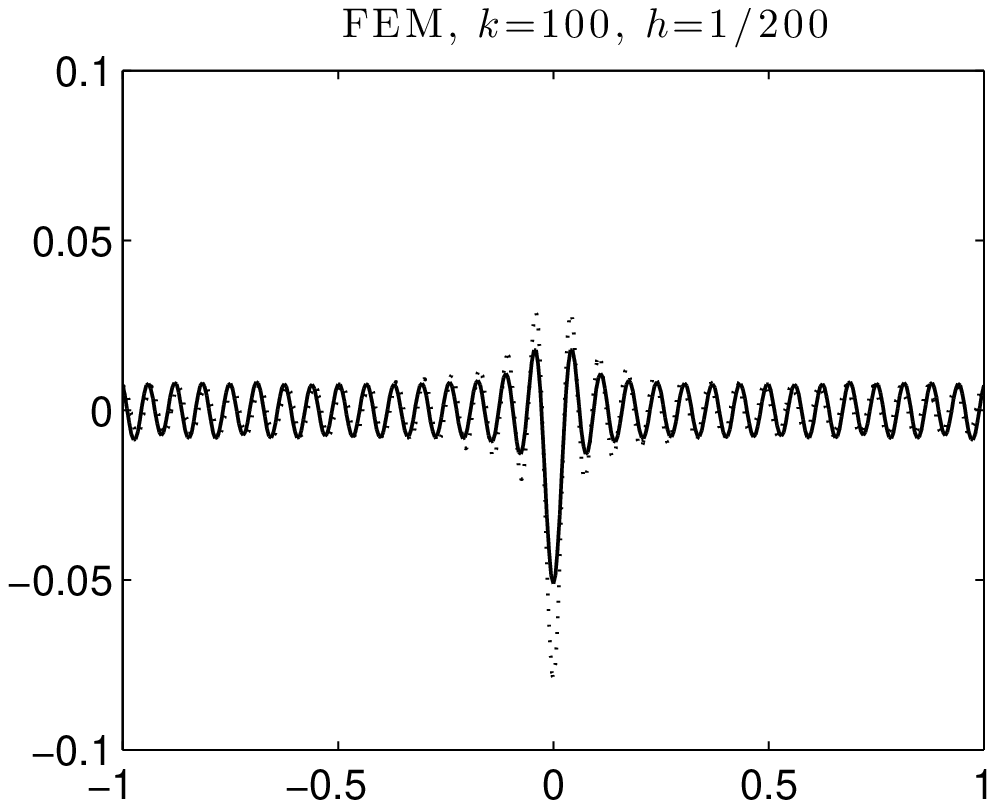}}
\caption{The traces of the CIP finite element solutions (left) with parameters given by
\eqref{e7.8}  and the finite element solutions (right) in the $xz$-plane
for $k=100$ and mesh sizes $h=1/50, 1/120$, and $1/200$, respectively.
The dotted lines give the trace of the exact solution in  the $xz$-plane.}
\label{ftrace}
\end{figure}

Table~\ref{table} shows the numbers of total DOFs needed for $30$\%
relative errors in $H^1$-seminorm for the finite element  interpolant, the CIP finite element solution  with parameters given by
\eqref{e7.8}, the finite element
solution, and the IPDG solution in \cite[Subsection 6.5]{fw09}, respectively.  The CIP-FEM needs less DOFs than the FEM does in all cases and much less for large wave number $k$. The IPDG method needs about six times as many total DOFs as the CIP-FEM to achieve the same accuracy but needs less DOFs than the FEM does when
$k=100, 200,$ and $300$.
\begin{table}[ht]
\centering
\begin{tabular}{|l|r|r|r|r|r|}
  \hline
    $k$ & 10 & 50 & 100 & 200 & 300 \\\hline
    Interpolation & 217& 5,167  & 20,419  & 81,181 & 182,287  \\\hline
    CIP-FEM & 217 & 6,487 & 35,971 & 239,419 & 754,507 \\\hline                  
    FEM & 397 & 30,301 & 229,357 & 1,804,201  & 6,053,461 \\\hline
    IPDG & 1,152 & 38,088 & 217,800 & 1,431,432 & 4,518,018\\\hline
  \end{tabular}
\caption{Numbers of total DOFs needed for
30\% relative errors
in $H^1$-seminorm for the finite element interpolant, the CIP finite element solution, the finite element solution, and the IPDG solution
in \cite[Subsection 6.5]{fw09},
respectively.}\label{table}
 \end{table}

\end{document}